\documentclass[10pt]{article}

\usepackage{authblk}
\usepackage{amsrefs}
\usepackage{amssymb}
\usepackage{amsfonts}
\usepackage{amsthm}
\usepackage{amsmath}
\usepackage{mathtools}
\usepackage{bbm}
\usepackage{bussproofs}
\usepackage{mathrsfs}
\usepackage{tikz, tikz-cd}
\usepackage{fullpage}
\usepackage[all,2cell]{xy}
\usepackage{todonotes}
\xyoption{2cell}
\usepackage{ stmaryrd } 
\usepackage{todonotes} 

\tikzset{%
	symbol/.style={%
		draw=none,
		every to/.append style={%
			edge node={node [sloped, allow upside down, auto=false]{$#1$}}}
	}
}

\usetikzlibrary{matrix,arrows}

\newtheorem{Theorem}{Theorem}

\newtheorem{proposition}[Theorem]{Proposition}
\newtheorem{lemma}[Theorem]{Lemma}
\newtheorem{corollary}[Theorem]{Corollary}

\theoremstyle{definition}
\newtheorem{example}[Theorem]{Example}
\newtheorem{remark}[Theorem]{Remark}
\newtheorem{definition}[Theorem]{Definition}

\DeclareMathOperator{\hfrak}{\mathfrak{h}}

\DeclareMathOperator{\Shv}{\mathbf{Shv}}

\DeclareMathOperator{\id}{id}
\DeclareMathOperator{\ob}{Ob}

\DeclareMathOperator{\Set}{\mathbf{Set}}

\DeclareMathOperator{\Topos}{\mathbf{Topos}}

\DeclareMathOperator{\C}{\mathbb{C}}
\DeclareMathOperator{\Z}{\mathbb{Z}}
\DeclareMathOperator{\N}{\mathbb{N}}

\DeclareMathOperator{\PP}{\mathbb{P}}

\DeclareMathOperator{\FSch}{\mathbf{FSch}}
\DeclareMathOperator{\Cring}{\mathbf{Cring}}

\DeclareMathOperator{\Grp}{\mathbf{Grp}}
\DeclareMathOperator{\Q}{\mathbb{Q}}

\DeclareMathOperator{\CalO}{\mathcal{O}}

\DeclareMathOperator{\Spec}{Spec}
\DeclareMathOperator{\Sch}{\mathbf{Sch}}
\DeclareMathOperator{\Dom}{Dom}
\DeclareMathOperator{\FinSet}{\mathbf{FinSet}}
\DeclareMathOperator{\CalF}{\mathcal{F}}
\DeclareMathOperator{\AffSch}{\mathbf{AffSch}}

\DeclareMathOperator{\Cat}{\mathbf{Cat}}

\DeclareMathOperator{\CalE}{\mathcal{E}}

\DeclareMathOperator{\Top}{\mathbf{Top}}

\DeclareMathOperator{\Fp}{\mathbb{F}_{\mathnormal{p}}}
\DeclareMathOperator{\Gr}{Gr}

\DeclareMathOperator{\LocRingSpac}{\mathbf{LRS}}
\DeclareMathOperator{\Spf}{Spf}
\DeclareMathOperator{\Codom}{Codom}

\DeclareMathOperator{\pfrak}{\mathfrak{p}}

\DeclareMathOperator{\A}{\mathbb{A}}
\DeclareMathOperator{\AFJ}{\mathcal{A}_{\mathnormal{F}}^{\mathnormal{J}}}
\DeclareMathOperator{\Fi}{Fi}
\DeclareMathOperator{\op}{op}
\DeclareMathOperator{\Ccat}{\mathscr{C}}
\DeclareMathOperator{\Dcat}{\mathscr{D}}
\DeclareMathOperator{\GL}{GL}

\let\emptyset\varnothing
\let\phi\varphi
\let\epsilon\varepsilon

\newcommand{\pushoutcorner}[1][dr]{\save*!/#1+1.2pc/#1:(1,-1)@^{|-}\restore} 
\newcommand{\pullbackcorner}[1][dr]{\save*!/#1-1.2pc/#1:(-1,1)@^{|-}\restore}

\UseTwocells

\numberwithin{Theorem}{section}

\begin{document}

\renewcommand{\thefootnote}{\fnsymbol{footnote}} 
\title{The Greenberg Functor is Site Cocontinuous}
\author{Geoff Vooys \\ gmvooys@gmail.com \\ Department of Mathematics and Statistics \\ University of Calgary, Calgary, Alberta, Canada}
\date{\today}
\footnotetext{2010 {\em Mathematics Subject Classification.} Primary 14F05; Secondary 18F20, 14B99, 18F10}
\renewcommand{\thefootnote}{\arabic{footnote}}

\maketitle
\begin{abstract}
In this paper we show that it is possible to define a topology on the category of formal schemes over a ring of $p$-adic integers such that the left adjoint of the Greenberg Transform is a site cocontinuous functor when we equip the  category of schemes over the residue field with the {\'e}tale topology. We show furthermore that this topology allows us to give an isomorphism between the corresponding fundamental groups, and use this isomorphism to show that it is possible to geometrize the quasicharacters of a $p$-adic torus by a local system on a formal scheme over the ring of $p$-adic integers.
\end{abstract}

\section{Introduction}


The Greenberg Transform, and its left adjoint (often called the Greenberg Functor), are two functors that are ubiquitous in arithmetic algebraic geometry. The Greenberg Transform itself first appeared in a proto-form in Serge Lang's thesis \cite{Lang} and was further explored and formally introduced by Marvin J.\@ Greenberg in \cite{Green1} and \cite{Green2}, where Greenberg showed that the transform that bears his name admits a left adjoint, as well as how to use the pair of functors to study torsors and cohomology. More recently, these functors have appeared in the book \cite{NeronModels} by Bosch, L{\"u}tkebohmert, and Raynaud on N{\'e}ron Models; in the work of Buium to study $p$-jet spaces (cf.\@ \cite{BuiumPJet}); in the work of Nicaise and Sebag to study motivic properties of (formal) schemes (cf.\@ \cite{NicaiseSebagMSIandWR} and \cite{NicaiseSebagMIand}); in the work of Cunningham and Roe in \cite{CunningRoe} to geometrize quasicharacters of $p$-adic tori; in the work of Yu in the study of smooth models as they are used in Bruhat-Tits theory (cf.\@ \cite{Yu}); in the work of Bertapelle and Tong to study the Picard group and the pro-algebraic structures of Serre (cf.\@ \cite{BertapelleTong}); and in the work of Bhatt and Scholze in defining and proving the representability of the positive $p$-adic loop group functor (cf.\@ Proposition 9.2 of \cite{BhattScholze}). 

While the Greenberg functor and the Greenberg Transform have given myriad tools with which to  study schemes and their arithmetic properties, especially by relating the mixed-characteristic case of formal schemes over $\Spec R$ to schemes over the residue field $\Spec k$, the Greenberg Transform was written in a pre-Grothendieck language, which made applying the rich theory around it more difficult. It was this issue that lead Bertapelle and Gonz{\'a}lez-Avil{\'e}s to revist and recast the Greenberg Transform into scheme (and formal scheme)-theoretic language of modern algebraic geometry in \cite{RevisGreen}. This significantly helped in the study of the Greenberg Transform, as it established many site-theoretic and geometric properties of the functor itself, and studied the Greenberg Transform in great detail. However, in contrast to the in-depth study and development of the Greenberg Transform, the theory surrounding the Greenberg functor is less developed.

In this paper, we work with and study the Greenberg functor and consider an application to local systems. In particular, we show the Greenberg functor is site-cocontinuous for certain Grothendieck topologies on $\Sch_{/\Spec k}$ and $\FSch_{/\Spec R}$, where $k$ is the residue field of a complete integral extension $R/\Z_p$ and $\FSch_{/\Spec R}$ is the category of formal schemes over $\Spec R$. More precisely, we will define topologies on the category $\FSch_{/\Spec R}$ for which the Greenberg functor $\hfrak:\Sch_{/\Spec k} \to \FSch_{/\Spec R}$ becomes cover-reflecting for the {\'e}tale and fppf topologies. This means that if $\lbrace \phi_i:X_i \to X \; | \; i \in I \rbrace$ is a collection of morphisms over $X$ in $\Sch_{/\Spec k}$, and if $\lbrace F\phi_i:FX_i \to FX \; | \; i \in I \rbrace$ generates a covering sieve over $FX$, then the set $\lbrace \phi_i:X_i \to X \; | \; i \in I \rbrace$ is a cover over $X$. This topology, which we will call $K$ for the moment (cf.\@ Definition \ref{Definition: Adhesive Site} for an explicit description), allows us to show that there is a canonical isomorphism of fundamental groups
\[
\pi^{\acute{E}t}_{1}(X,x) \cong \pi_{1}^{K}(\hfrak X, \hfrak x)
\]
for any scheme $X$ over $\Spec k$. In particular, from this group isomorphism, we show that there is an isomorphism of the category of {\'e}tale local systems over a $k$-scheme $X$ with the category of local systems in the $K$ topology over $\hfrak X$. 

\begin{Theorem}[cf.\@ Theorem \ref{Theorem: Gometrization Theorem}]
	For any group scheme $G$ over $\Spec k$ with geometric point $g$ of $G$, if $\hfrak$ is the Greenberg functor, then there is an isomorphism of categories
	\[
	\mathbf{Rep}(\pi_1^{\acute{E}t}(G,g)) \cong \mathbf{Rep}(\pi_1^{K}(\hfrak G, \hfrak g)).
	\]
	In particular, this gives an isomorphism of  categories
	\[
	\mathbf{Loc}_{\acute{E}t}(G) \cong \mathbf{Loc}_{K}(\hfrak G).
	\]
\end{Theorem}

Our study of the Greenberg functor is motivated in two directions: 
\begin{enumerate}
	\item The first motivation comes from wishing to better understand the Greenberg Transform and a desire to use it in arithmetic algebraic geometry, since understanding a right adjoint functor is to understand its left adjoint, and dually. Furthermore, a deeper understanding of this left adjoint and how it affects Grothendieck topologies on its domain will lead to a comparison between Grothendieck topologies on schemes over fields of prime characteristic and Grothendieck topologies on categories of formal schemes over $p$-adically complete rings.
	\item The second motivation comes from \cite{CunningRoe} in which the authors begin with a torus $T$ over  a $p$-adic field $F$, take its N{\'e}ron model $N_T$ over its ring of integers $\CalO_F$, and then take the Greenberg Transform of $N_T$, to define a scheme $X$ over $\Spec k$ for which there is an identification of groups
	\[
	T(F) \cong N_T(\CalO_F) \cong X(k).
	\]
	By then taking the category of {\'e}tale local systems over $X$ and using the Trace of Frobienius, they arrive at the following: Any quasicharacter of $T(F)$ is geometrized as the Trace of Frobenius of some {\'e}tale local system $\mathcal{L}$. This causes one to ask whether it is possible to construct a category of local systems in some topology over a formal scheme over $\Spec \CalO_F$ for which this new local system geometrizes the quasicharacter $\chi:T(F) \to \C^{\ast}$ directly.
\end{enumerate}
In an attempt to address both of the considerations above, we are lead to study a class of functors, which we call geometrically adhesive (cf.\@ Definition \ref{Definition: Geometrically Adhesive}); Corollary \ref{Cor: Greenberg functor is adhesive} shows that the Greenberg functor is geometrically adhesive. With this notion, we build a topology which allows the transfer of sites from categories of schemes to categories of formal schemes in the sense that the fundamental groups are preserved by the transfer. This is Theorem \ref{Theorem: Isomorphism of Fundamental Groups}, which is the main theorem of the paper:
\begin{Theorem}[cf.\@ Theorem \ref{Theorem: Isomorphism of Fundamental Groups}]
Let $F:\Ccat \to \Dcat$ be a geometrically adhesive functor and assume that there is a Grothendieck topology $J$ on $\Ccat$ for which $\Shv(\Ccat,J)_{lcf}$ is a Galois category. Then there is a topology $K$ on $\Dcat$ and a functor $F^{\ast}:\Shv(\Dcat,K)_{lcf} \to \Shv(\Ccat,J)_{lcf}$ such that if $F^{\ast}$ is fully faithful, there is an isomorphism of fundamental groups
\[
\pi_1^{J}(X,x) \cong \pi_1^{K}(FX,Fx)
\]
for all objects $X$ of $\Ccat$ and geometric points $x$ of $X$.
\end{Theorem}

As an application, we show that the pullback functor $\hfrak^{\ast}:\Shv(\FSch_{/\Spec R},J) \to \Shv(\Sch_{/\Spec k},J)$ is fully faithful (cf.\@ Lemma \ref{Lemma: Greenberg pullback is fully faithful}). Furthermore, we derive:
\begin{corollary}
If $J$ is a topology on $\Sch_{/\Spec k}$ for which $\Shv(\Sch_{/\Spec k},J)$ is a Galois category and if $\hfrak$ is the Greenberg functor, then for any $k$-scheme $X$ there is an isomorphism of fundamental groups
\[
\pi^{J}_{1}(X,x) \cong \pi_{1}^{K}(\hfrak X, \hfrak x).
\]
\end{corollary}

\section{An Introduction to Geometrically Adhesive Functors, with Motivation}\label{Section: Intro to adhesion}
We begin by recalling the full-level Greenberg Transform for a complete field extension $K/\Q_p$; cf.\@ \cite{RevisGreen} for a modern account and \cite{Green1} and \cite{Green2} for the introduction of the functors. Begin by assuming that $R/\Z_p$ is an absolutely unramified extension ring of $\Z_p$ with residue field $k/\Fp$ and fraction field $K/\Q_p$. Furthermore, if $n \in \N$, define the scheme
\[
S_n := \Spec \frac{R}{\pfrak^{n+1}}
\]
so that
\[
\Spec R \cong \Spec\left(\lim_{\longleftarrow} \frac{R}{\pfrak^{n+1}}\right) \cong \lim_{\longrightarrow}\Spec\frac{R}{\pfrak^{n+1}} = \lim_{\longrightarrow} S_n,
\]
where the limit is calculated in the category $\Cring$ of commutative rings with identity and the colimit is calculated in the category $\AffSch$ of affine schemes. Then the full-level Greenberg Transform is a functor $\Gr:\FSch_{/\Spec R} \to \Sch_{/\Spec k}$, where $\FSch_{/\Spec R}$ is the category of formal schemes over the scheme $\Spec R$, has the left adjoint $\mathfrak{h}:\Sch_{/\Spec k} \to \FSch_{/\Spec R}$, i.e., we have an adjucntion diagram
\[
\begin{tikzcd}
\Sch_{/\Spec k} \ar[r,bend left,"\hfrak",""{name=A, below}] & \FSch_{/\Spec R}\ar[l,bend left,"F^{\ast}",""{name=B,above}] \ar[from=A, to=B, symbol=\dashv]
\end{tikzcd}
\]
In order to give a particularly concrete description of $\hfrak$, we now assume that the ramification degree $e = [K:\Q_p]/[k:\Fp] = 1$, i.e., $K$ is unramified over $\Q_p$. The left adjoint $\mathfrak{h}$ then can be seen as the colimit of the functors
\[
h_n:\Sch_{/\Spec k} \to \Sch_{/S_n} \hookrightarrow \FSch_{/\Spec R}
\]
where each $h_n$ is defined by, for each scheme $X = (\lvert X \rvert, \CalO_X)$ in $\Sch_{/\Spec k}$,
\[
h_n\lvert X \rvert := \lvert X \rvert
\]
and
\[
h_n\CalO_X := W_{n+1}\CalO_X
\]
where $W_n\CalO_X$ is the sheaf of length $n$ Witt Vectors in $\CalO_X$; note that the action of $W_{n+1}\CalO_X$ is given by, for each open set $U \subseteq \lvert X \rvert$,
\[
\bigg(W_{n+1}\CalO_{X}\bigg)(U) =  W_{n+1}\bigg(\CalO_X(U)\bigg),
\]
where $W_{n+1}(\CalO_X(U))$ is the ring of length $n+1$ Witt Vectors with coefficients in $\CalO_X(U)$. It is worth observing that na{\"i}vely these functors do not admit a colimit; however, since there is a natural map
\[
S_{m} \to S_{n}
\]
whenever $m \leq n$, as it comes from the correspondence $\Sch(S_m,S_n) \cong \Cring(R/\pfrak^{n+1}, R/\pfrak^{m+1})$;
in this way a scheme over $S_m$ is a scheme over $S_n$ by simply post-composing by map $S_m \to S_n$, which itself allows us to regard each functor $h_m:\Sch_{/\Spec k} \to \Sch_{/S_{m}}$ as a functor $h_m:\Sch_{/\Spec k} \to \Sch_{/S_n}$. Through this process, we derive a natural transformation $h_n \to h_{n+1}$ for all $n \in \N$. Taking the colimit of these functors $h_n$, together with all necessary embeddings of categories $\Sch_{/S_n} \to \FSch_{/\Spec R}$, gives the desired colimit
\[
\mathfrak{h}:\Sch_{/\Spec k} \to \FSch_{/\Spec R}.
\]
From these definitions a routine calculation implies that for each scheme $X = (\lvert X \rvert, \CalO_X)$ in $\Sch_{/\Spec k}$, we have that
\[
\mathfrak{h}\lvert X \rvert = \lvert X \rvert
\]
and
\[
\mathfrak{h}\CalO_X = W\CalO_X,
\]
where $W\CalO_X$ is the sheaf of Witt Vectors on $\CalO_X$. A standard gluing argument then allows one to show that if $\lbrace U_i \; | \; i \in I \rbrace$ is an open gluing of a scheme $X$ in $\Sch_{/\Spec k}$, then
\[
\mathfrak{h}\left(\bigcup_{i \in I} U_i\right) = \mathfrak{h}X = \bigcup_{i \in I} \mathfrak{h}U_i;
\]
such a geometric condition is very powerful and is easier to work with than the mystical and difficult to understand Greenberg Transform itself. In particular, we call such a functor {\em geometrically adhesive} because it preserves all possible gluing!
\begin{remark}
The union written above is meant simply as a short hand for a specific colimit. Explicitly, for locally ringed spaces $U$ and $V$, we write
\[
U \cup V := U \coprod_{X} V
\]
as the gluing pushout of along their intersection $X \to U$ and $X \to V$. Similarly, we write $U \cap V$ to denote the pullback of $U$ and $V$ along the morphisms $U \to U \cup  V$ and $V \to U \cup V$. In particular, the diagram
\[
\xymatrix{
U \cap V \pullbackcorner \ar[r] \ar[d] & U \ar[d] \\
V \ar[r] & \pushoutcorner U  \cup V	
}
\]
in $\LocRingSpac$ is simultaneously a pushout and pullback diagram.
\end{remark}
\begin{remark}
If the algebraic field extension $F/\Q_p$ has nontrivial ramification, say as in the quadratic extension $\Q_p(\sqrt{p})$, then there is an analogous story to the one told above that may be used to construct the Greenberg Transform and its left adjoint $\hfrak$. It essentially the same construction, save for now we take products of the $W_n\CalO_X$ sheaves to build up the Eisenstein equation of ramification termwise, and then build up our Witt Vector structure with these ramifications in mind. Explicit details may be found in \cite{RevisGreen}, but will not be particularly relevant for this paper, save for in providing intuition for the explicit case to which we wish to apply our results.
\end{remark}

We will now codify some notation useful as we proceed in the section. In particular, let  $\Ccat$ and $\Dcat$ be subcategories of the category $\LocRingSpac$ of locally ringed spaces or slice categories of subcategories of $\LocRingSpac$.
\begin{definition}\label{Definition: Geometrically Adhesive}
A functor $F:\Ccat \to \Dcat$ is said to be {\em geometrically adhesive} if, whenever $U = (\lvert U \rvert, \CalO_U) \in \ob\Ccat_{/X}$ has a collection of open subobjects $\lbrace V_j \; | \; j \in J \rbrace$ with
\[
\bigcup_{j \in J} V_j = V,
\]
then
\[
F\left(\bigcup_{j \in J}V_j\right) = \bigcup_{j \in J} FV_j.
\]
\end{definition}
\begin{remark}
For the reader familiar with the work of Lack, Cockett, et al.\@ (cf.\@ \cite{RobinGeoff}, \cite{RobinJohnGeoff}, \cite{RobinLackIII}, amongst others), one may notice the similarity of a join restriction functor between join restriction categories, or perhaps more readily with adhesive functors between adhesive categories (cf.\@ \cite{AdhesiveQuasiAdhesive}, for instance), and our given definition of a geometrically adhesive functor above. There are similarities, especially in the fact that these are all getting towards some sort of manifold-type construction and intuition, but we have given the definition above for the following reasons:
\begin{itemize}
	\item The  definition of a geometrically adhesive functor is intimately tied to the sheaves that come equipped with a locally ringed space, and geometrically adhesive functors are built to bring this perspective to mind and to task.
	\item We are not giving any extra attention to restriction or join properties of the categories $\Ccat$ or $\Dcat$, if nontrivial versions of said structures even exist, save for potentially in incidental circumstances.
	\item The  definition of join restriction functors is tied to the abstract differential geometry of Cartesian differential categories. Because we are focusing only on the algebro-geometric aspects induced by these functors, we provide an algebraic dual to the more analytic perspective afforded by join restriction functors.
	\item These are different from the adhesive functors of Lack (save for perhaps morally) in the following way: A geometrically adhesive functor only asks to preserve gluings of open subobjects between arbitrary categories of locally ringed spaces, while adhesive functors in the sense of Lack are functors between adhesive categories that preserve pushouts against arbitrary monomorphisms. In particular, Proposition \ref{Proposition: Adhesive need not be cocont} shows that in general these are different notions.
	\item Proposition \ref{Proposition: Adhesive need not be cocont} shows, in addition to differentiating adhesive functors from geometrically adhesive functors, that geometrically adhesive functors are not necessarily cocontinuous.
\end{itemize}
\end{remark}

\begin{remark}
Not all functors are geometrically adhesive. For instance, consider the functor $\Gamma:\Sch \to \Sch$ given by
\[
\Gamma(X) \mapsto \Spec(\CalO_X(X)).
\]
This functor is not geometrically adhesive because it destroys nonaffine schemes and their affine gluings. In particular, recall that if $X$ is affine if and only if
\[
\Gamma(X) \cong \Gamma(\Spec A) = \Spec(\CalO_A(\lvert \Spec A\rvert)) = \Spec A.
\]
To see that this is not geometrically adhesive, consider the scheme $\mathbb{P}^1_{\Z}$ with the gluing
\[
\xymatrix{
 & \PP_{\Z}^{1} & \\
\A_{\Z}^1 \ar[ur] & & \A_{\Z}^{1} \ar[ul] \\
\Spec\Z[x,x^{-1}] \ar[u] \ar@<.5ex>[rr] & & \Spec\Z[y,y^{-1}] \ar@<.5ex>[ll]^{\cong} \ar[u]
}
\]
where the bottom isomorphism is induced by the $\Cring$ morphism $\phi:\Z[x,x^{-1}] \to \Z[y,y^{-1}]$ given by $x \mapsto y^{-1}$. Then we find that 
\[
\Gamma(\PP_{\Z}^1) = \Spec(\CalO_{\PP^1_{\Z}}(\lvert \PP^{1}_{\Z}\rvert)) = \Spec \Z
\]
so taking $\Gamma$ of the whole diagram gives the commuting diagram
\[
\xymatrix{
 & \Spec \Z & \\
\A^{1}_{\Z} \ar@{-->}[ur]^{\exists!} & & \A^1_{\Z} \ar@{-->}[ul]_{\exists!} \\
\ar[u] \Spec \Z[x,x^{-1}] \ar@<.5ex>[rr] & & \Spec \Z[y,y^{-1}] \ar@<.5ex>[ll]^{\cong} \ar[u]
}
\]
which is not a gluing diagram because, amongst other reasons, $\A^1_{\Z}$ is not a subscheme of $\Spec \Z$.
\end{remark}
\begin{example}
Let $\Dcat$ be a subcategory of $\LocRingSpac$ such that $\Spec 0 \in\ob\Dcat$. Then the functor $F:\Ccat \to \Dcat$ defined by $FU := \Spec 0$ is geometrically adhesive.
\end{example}
\begin{proposition}
Adhesive functors need not reflect isomorphisms.
\end{proposition}
\begin{proof}
Let $F$ be the functor given in the above example and let $\Ccat := \Sch$. Consider the schemes $U_i, U_j := \mathbb{A}_{\Z}^1$ with gluing
\[
U_i \cup  U_j = \widetilde{U} := \mathbb{P}^1_{\Z}
\]
and define $U := \Spec \Z$. We then calculate that
\[
F(U) = F(\Spec \Z) = \Spec 0 = F(\mathbb{A}_{\mathbb{Z}}^1 \cup \mathbb{A}_{\Z}^1) = F(\mathbb{P}_{\Z}^1) = F(\widetilde{U})
\]
but, of course, $U = \Spec \Z \not\cong \mathbb{P}^1_{\Z} = \widetilde{U}$.
\end{proof}
\begin{proposition}
	Adhesive functors need not be faithful.
\end{proposition}
\begin{proof}
	The ``affinization'' functor $F$ above does the job for sufficiently chosen categories $\Ccat$ of locally ringed spaces. In particular, take $\Ccat = \Sch$.
\end{proof}
\begin{proposition}\label{Proposition: Adhesive need not be cocont}
Geometrically adhesive functors need not be cocontinuous. In particular, geometrically adhesive functors need not preserve gluings along closed subobjects.
\end{proposition}
\begin{proof}
As above, the proof is by example. Consider the ring
\[
A := \lbrace (a,f) \; | \; a \in \Z_p, f \in \Z_p[x], a \equiv f(0) \mod{p} \rbrace = \Z_p \times_{\Fp} \Z_p[x]
\]
and define the set of objects
\[
\Ccat_0 := \lbrace \Spec A, \Spec \Z_p, \Spec \Fp, \A_{\Z_p}^{1}, \mathbb{P}^{1}_{\Z_p}, \Spec \Z_p[x,x^{-1}]\rbrace
\]
with morphism set
\[
\Ccat_1 := \lbrace f \; | \; \exists X, Y \in \Ccat_0.\, f \in \Sch(X,Y) \rbrace
\]
and take $\Ccat$ to be the  category $(\Ccat_0,\Ccat_1)$. Note that the only object with nontrivial open subobjects is $\PP_{\Z_p}^{1}$, which is the gluing of two copies of $\A_{\Z_p}^{1}$ along the isomorphism $\Spec \Z_p[x,x^{-1}] \to \Spec \Z_p[x,x^{-1}]$ given by $x \mapsto x^{-1}$. Furthermore, observe that the object $\Spec A$ is a gluing along a closed subscheme by an argument of \cite{SchwedeClosedGlue}, i.e., $\Spec A$ is the pushout
\[
\xymatrix{
	\Spec \Fp \ar[r]^{s} \ar[d]_{(p,x)} & \Spec \Z_p \ar[d]^{\iota_1} \\
	\A_{\Z_p}^{1} \ar[r]_{\iota_2} & \pushoutcorner \Spec A
}
\]
where the closed immersion $\Spec \Fp \to \Spec \Z_p$ picks out the special fibre and the closed immersion $\Spec \Fp \to \A^1_{\Z_p}$ picks out the closed point $(p,x)$.

Define a functor $F:\Ccat \to \Sch$ as follows: If $X \in \Ccat_0$, define
\[
FX := \begin{cases}
\Spec \Z_p & {\rm if}\, X = \Spec \Fp; \\
X & {\rm else}
\end{cases}
\]
and if $f \in \Ccat_1$,  define
\[
F(f) := \begin{cases}
f & {\rm if}\, \Dom f \ne \Spec \Fp; \\
\id_{\Spec \Z_p}& {\rm if}\, f = \id_{\Spec \Fp}; \\
\phi & {\rm if}\, \Dom f = \Spec \Fp,  f = \phi \circ s, \phi:\Spec \Z_p \to X, X \ne \Spec \Fp;
\end{cases}
\]
this fully defines the functor because for any $X \in \Ccat_0$, if $X \ne \Spec \Fp$ then any map $\Spec \Fp \to X$ factors through the special fibre of $\Spec \Z_p$, i.e., there exists a morphism $g:\Spec \Z_p \to X$ making the diagram
\[
\xymatrix{
\Spec \Fp \ar[rr]^{f} \ar[dr]_{s} & & X \\
 & \Spec \Z_p \ar[ur]_{\phi}	
}
\]
commute in $\Sch$. In particular, with this the verification that $F$ is a functor is trivial and omitted.

The functor $F$ just perserves the gluing of $\PP^1_{\Z_p}$ by construction, and since this is the only nontrivial open subobject gluing, this shows that $F$ is geometrically adhesive. On the other hand, $F$ sends the pushout diagram defining $\Spec A$ to the diagram, where $x:\Spec \Z_p \to \A_{\Z_p}^{1}$ is the spectrum of the map $\Z_p[x] \to \Z_p$ corresponding to evaluation at $x = 0$,
\[
\xymatrix{
	\Spec \Z_p \ar[d]_{x} \ar@{=}[r] & \Spec \Z_p \ar[d]^{\iota_1} \\
	\A_{\Z_p}^{1} \ar[r]_{\iota_2}& \Spec A
}
\]
which is evidently not a pushout in $\Sch$; explicitly, the pushout of $\Spec \Z_p \xleftarrow{=} \Spec \Z_p \xrightarrow{x} \A_{\Z_p}^{1}$ is $\Spec B$, where $B$ is the ring
\[
B = \lbrace (a,f) \; | \; a  \in \Z_p, f \in \Z_p[x], f(0) = a \rbrace = \Z_p \times_{\Z_p} \Z_p[x] \cong \Z_p[x].
\] 
But then $F$ does not preserve the closed gluing of $\Spec A$ and does not preserve the pushout diagram. This proves that geometrically adhesive functors need not be cocontinuous or preserve gluings along closed subobjects.
\end{proof}
\begin{proposition}
Adhesive functors need not be full.
\end{proposition}
\begin{proof}
Consider the functor $\hfrak$ extended to all schemes $X$ via the same assignment, i.e., via $\hfrak X = (\lvert \hfrak X \rvert, \hfrak\CalO_X)$ where
\[
\lvert \hfrak X \rvert = \lvert X \rvert
\]
and the sheaf $\hfrak \CalO_X$ is defined by, for all $U \subseteq \lvert X \rvert$ open,
\[
\hfrak \CalO_X(U) = W(\CalO_X(U))
\]
where $W$ is the $p$-typical Witt vector functor a l{\'a} Borger; see \cite{Borger1} and \cite{Borger2} for details. A routine calculation (when we identify $\FSch_{/\Spec R}$ as the ind-closure of $\Sch_{/\Spec R}$ and then give all affine schemes the trivial filtration) shows that
\[
\FSch(\Spec \Z_p, \Spec \Fp) \cong \Cring(\Fp,\Z_p) = \emptyset.
\]
However, since there is a unique continuous map $\tau:\lbrace \eta, s \rbrace \to \lbrace \ast \rbrace$, it is the case that we could have a sheaf morphism $\tau^{\sharp}:\hfrak\CalO_{\Fp} \to \tau_{\ast}\hfrak\CalO_{\Z_p}$. Explicitly we note that
\[
\hfrak\CalO_{\Z_p}(U) = W(\CalO_{\Z_p}(U)) = \begin{cases}
\Z_p \oplus \mathfrak{a} & {\rm if}\, U  = \lvert \Spec \Z_p \rvert; \\
\Q_p^{\N} & {\rm if}\, U = \lbrace \eta \rbrace; \\
0 & {\rm if}\, U = \emptyset
\end{cases}
\]
where $\mathfrak{a} = \bigoplus_{n \geq 1} p^n\Z_p$ and the multiplication in $\Z_p \oplus \mathfrak{a}$ is by the rule
\[
(a,x)(b,y) = (ab, ay + bx + xy).
\]
The morphism $\phi:\Z_p \to \Z_p \oplus \mathfrak{a}$ given by $a \mapsto (a,0)$ is  a ring homomorphism. On the other hand we see that for all $U \subseteq \lvert \Spec \Fp \rvert$ open,
\[
\hfrak\CalO_{\Fp}(U) = W(\CalO_{\Fp}(U)) = \begin{cases}
\Z_p & {\rm if}\, U = \lbrace \ast \rbrace; \\
0 & {\rm if}\, U = \emptyset.
\end{cases}
\]
Then we can define a sheaf morphism $\tau^{\sharp}:\hfrak\CalO_{\Fp} \to \tau_{\ast}\hfrak\CalO_{\Z_p}$ because such a morphism of locally ringed spaces only sees global sections and the diagram
\[
\xymatrix{
\Z_p \ar@{-->}[d]_{\exists!} \ar[r]^{\phi} & \Z_p \oplus \mathfrak{a} \ar@{-->}[d]^{\exists!} \\
0 \ar@{-->}[r]_{\exists!} & 0	
}
\]
commutes; moreover, it is not hard to see that this is the only such sheaf map between $\hfrak\CalO_{\Fp}$ and $\tau_{\ast}\hfrak\CalO_{\Z_p}$. Thus we find that
\[
\FSch(\hfrak\Spec \Z_p, \hfrak\Spec \Fp) = \lbrace (\tau,\tau^{\sharp}) \rbrace \ne \emptyset = \Sch(\Spec \Z_p, \Spec \Fp).
\]
This shows that $\hfrak$ is not full when extended to a functor $\hfrak:\Sch \to \FSch$.
\end{proof}

We now move from these definitions and remarks to give some properties of geometrically adhesive functors. In particular we will show that if $U$ is an open subobject of $V$ in $\Ccat$ (that is there is an open immersion $i:U \to V$) and if $F:\Ccat \to \Dcat$ is geometrically adhesive, then $FU$ is an open subobject of $FV$ (more explicitly, the map $Fi:FU \to FV$ is an open immersion). 
\begin{proposition}\label{Prop: Adhesive Functors Preserve Open Immersions}
Let $i:U \to V$ be an open immersion in $\Ccat$ and assume that $F:\Ccat \to \Dcat$ is geometrically adhesive. Then $Fi:FU \to FV$ is an open immersion.
\end{proposition}
\begin{proof}
Begin by recalling that the fact that $F$ is geometrically adhesive implies that if $\lbrace V_i \to V \; | \; i \in I \rbrace$ is an open covering of $V$, then
\[
FV = F\left(\bigcup_{i \in I} V_i\right) = \bigcup_{i \in I} FV_i.
\]
Thus it follows that each $FV_i$ remains a subobject of $FV$, and hence we conclude that $\lvert FV^{\prime} \rvert \subseteq \lvert FV \rvert$ whenever $V^{\prime}$ is an open subobject of $V$. Therefore we have that $FU$ is a subobject of $FV$.

Finally, to see that $FU$ is an open subobject of $FV$, we note that since $F$ is a functor between (slice) categories of locally ringed spaces, $F$ takes sheaves to sheaves and hence is at least Zariski continuous. Now consider the map $i:U \to V$ and observe that we can deduce, from the functoriality of $F$ and the pushforward/pullback adjunction $i^{-1} \dashv i_{\ast}:\Shv(V) \to \Shv(U)$, that the following deduction may be made: First, observe that since $i$ is an open immersion, there is an isomorphism of $U$-sheaves $i^{\flat}:i^{-1}\CalO_V \cong \CalO_U$; pushing this through the adjunction gives rise to the sheaf morphism
\[
i^{\sharp}:\CalO_V \to i_{\ast}\CalO_U.
\]
Applying the functor $F$ on sheaves then gives the sheaf morphism
\[
F(i^{\sharp}):F\CalO_V \to (Fi)_{\ast}F\CalO_U
\]
which is equivalent by the functoriality of $F$ to the sheaf morphism
\[
F(i)^{\sharp}:\CalO_{FV} \to F(i)_{\ast}\CalO_{FU}.
\]
Passing this through the adjunction $F(i)^{-1} \dashv F(i)_{\ast}$ sends the morphism $F(i)^{\sharp}$ to  the map
\[
F(i)^{\flat}:F(i)^{-1}\CalO_{FV} \to \CalO_{FU}
\]
which is equivalent, again by the functoriality of $F$, to the map
\[
F(i^{\flat}):(Fi)^{-1}F\CalO_V \to F\CalO_U.
\]
Now, because $i^{\flat}$ is an isomorphism it follows that $F(i^{\flat}) = F(i)^{\flat}$ is as well. Therefore $(Fi)^{-1}\CalO_{FV} \to \CalO_{FU}$ is an isomorphism of $FU$-sheaves, and hence the map $Fi:FU \to FV$ is an open immersion. This completes the proof of the proposition.
\end{proof}

We now proceed to show that geometrically adhesive functors preserve pullbacks defined by open gluings, i.e., that such functors perserve intersections. This will be necessary because it will allow us to show that geometrically adhesive functors preservve pullbacks along covers induced by pretopologies. As we proceed with this, we assume the following:
\begin{enumerate}
	\item The category $\Ccat$ has the property that if $U$ and $V$ are objects of $\Ccat$, then the union space $U \cup V$ and intersection space $U \cap V$ exists in $\Ccat$.
\end{enumerate}
In this way we can regard any two objects as open subobjects of some larger object, and we can view their intersection as an open subspace of $U, V,$ and $U \cup V$.
\begin{Theorem}
If $U$ and $V$ are objects in $\Ccat$ and $F:\Ccat \to \Dcat$ is geometrically adhesive, then
\[
F(U \cap V) = FU \cap FV.
\]
\end{Theorem}
\begin{proof}
Let $U$ and $V$ be objects of $\Ccat$ and assume that there is are open immersions $U, V \to W$, for some $W \in \ob\Ccat$. However, since $U$ and $V$ are glued along their intersection, it suffices to prove the proposition for the locally ringed space
\[
W := U \coprod_{U \cap V} V = U \cup V,
\]
where the pushout $U \coprod_{U \cap V} V$ is the gluing of $U$ and $V$ along the subspace $U \cap V$; note also that in this definition, since $U$ and $V$ are open in $W$, $U \cap V$ is open and the immersion $U \cap V \to W$ factors as
\[
\xymatrix{
 & V \ar[dr]^{i_{V, W}} & \\
U \cap V \ar[ur]^{i_{U \cap V,V}} \ar[dr]_{i_{U \cap V, U}} \ar[rr] & & W \\
 & U \ar[ur]_{i_{U,W}}	
}
\]
where each of the $i$ morphisms are open immersions.

We will first show that the canonical map $F(U \cap V) \to FU \cap FV$ is an open immersion. Begin with the observations that
\[
FW = F(U \cup V) = FU \cup FV = FU \coprod_{FU \cap FV} FV
\]
and that $U \cap V$ is the pullback
\[
\xymatrix{
U \cap V \pullbackcorner \ar[rr]^-{i_{U \cap V, V}} \ar[d]_{i_{U \cap V, U}} & & V \ar[d]^{i_{V,W}} \\
U \ar[rr]_-{i_{U,W}} & & W	
}
\]
in $\Ccat$. Thus we get that $FU \cap FV$ is the pullback
\[
\xymatrix{
FU \cap FV \pullbackcorner \ar[rr]^-{i_{FU \cap FV, FV}} \ar[d]_{i_{FU \cap FV,FU}} & & FV \ar[d]^{i_{FV,FW}} \\
FU \ar[rr]_-{i_{FU,FW}} & & FW	
}
\]
and so applying $F$ gives rise to a unique morphism $\theta:F(U \cap V) \to FU \cap FV$ making the diagram
\[
\xymatrix{
F(U \cap V) \ar@{-->}[dr]_{\exists!\theta} \ar@/^/[drrr]^{Fi_{U \cap V, V}} \ar@/_/[ddr]_{Fi_{U \cap V, U}} & & & \\
	& FU \cap FV \ar[rr]_-{i_{FU \cap FV, FV}} \ar[d]^{i_{FU \cap FV, FU}} & & FV \ar[d]^{i_{FV,FW}} \\
	& FU \ar[rr]_-{i_{FU,FW}} & & FW
}
\]
commute in $\Dcat_{/Y}$. Moreover, since $F$ preserves open immersions by Proposition \ref{Prop: Adhesive Functors Preserve Open Immersions}, the morphisms $Fi_{U \cap V,V}$ and $Fi_{U \cap V, U}$ are both open immersions in $\Dcat_{/Y}$. Thus, from the equations
\[
i_{FU \cap FV, FU} \circ \theta = Fi_{U \cap V,U}
\]
and
\[
i_{FU \cap FV, FV} \circ \theta = Fi_{U \cap V, V}
\]
it follows that $\theta$ is both monic and open, and hence an open immersion.

We will now show that
\[
FU \coprod_{FU \cap FV} = FU \cup  FV = FW = FU \coprod_{F(U \cap V)} FV,
\]
as this will imply that $F(U \cap V) \cong FU \cap FV$; the fact that $F(U \cap V)$ is an open subobject of $FU \cap FV$ will then complete the proof. To do this assume that there is an object $S$ of $\Dcat_{/Y}$ and that there are morphisms $\phi:FU \to S$ and $\psi:FV \to S$ such that the rectangle
\[
\xymatrix{
FU \ar[rr]^{\phi} 	& & S \\
FU \cap FV \ar[u]^{i_{FU \cap FV, FU}} \ar[rr]_-{i_{FU \cap FV,FV}} & & FV \ar[u]_{\psi}	
}
\]
commutes. Then there exists a unique morphism $\rho:FW \to S$ making the diagram
\[
\xymatrix{
 &	& & S \\
FU \ar@/^/[urrr]^{\phi} \ar[rr]^-{i_{FU,FW}} & & FW \ar@{-->}[ur]^{\exists!\rho} & \\
FU \cap FV \ar[u]^{i_{FU \cap FV, FU}} \ar[rr]_-{FU \cap FV, FV} & & FV \ar[u]^-{i_{FV,FW}} \ar@/_/[uur]_{\psi}
}
\]
commutes. Now observe that since the open immersion 
\[
i_{F(U \cap V), FU \cap FV}:F(U \cap V) \to FU \cap FV
\]
makes the diagram
\[
\xymatrix{
	& & & S \\
FU \ar[rr]^-{i_{FU,FW}} \ar@/^/[urrr]^{\phi} & & FW \ar[ur]^{\rho} & \\
 & FU \cap FV \ar[ul]^{i_{FU \cap FV, FU}} \ar[dr]_{i_{FU \cap FV, FV}} & \\
F(U \cap V) \ar[ur] \ar[uu]^{i_{F(U \cap V),FU}} \ar[rr]_-{i_{F(U \cap V),FV}} & & FV \ar[uu]^-{i_{FV,FW}} \ar@/_/[uuur]_{\psi}
}
\]
commute, where the arrow $F(U \cap V) \to FU \cap FV$ is the open immersion $i_{F(U \cap V),FU \cap FV}$; thus using that
\[
i_{F(U \cap V),FU} = i_{FU \cap FV,FU} \circ i_{F(U \cap V), FU \cap FV}
\]
and
\[
i_{F(U \cap V),FV} = i_{FU \cap FV,FV} \circ i_{F(U \cap V),FU \cap FV}
\]
as open immersions, we have that the diagram
\[
\xymatrix{
	& & & S \\
FU \ar[rr]^-{i_{FU,FW}} \ar@/^/[urrr]^{\phi} & & FW \ar[ur]^{\rho} \\	
F(U \cap V) \ar[u]^{i_{F(U \cap V),FU}} \ar[rr]_-{i_{F(U \cap V),FV}} & & FV \ar@/_/[uur]_{\psi} \ar[u]^{i_{FV,FW}}
}
\]
commutes as well. Thus, to show that $FW = FU \coprod_{F(U \cap V)} FV$ as well, we need only show that $\rho$ is the unique such map making this diagram commute.

With this in mind, assume that there exists a morphism $\sigma:FW \to S$ making the diagram
\[
\xymatrix{
	& & & S \\
	FU \ar[rr]^-{i_{FU,FW}} \ar@/^/[urrr]^{\phi} & & FW \ar[ur]^{\sigma} \\	
	F(U \cap V) \ar[u]^{i_{F(U \cap V),FU}} \ar[rr]_-{i_{F(U \cap V),FV}} & & FV \ar@/_/[uur]_{\psi} \ar[u]^{i_{FV,FW}}	
}
\]
commute as well. But then we derive that the diagram
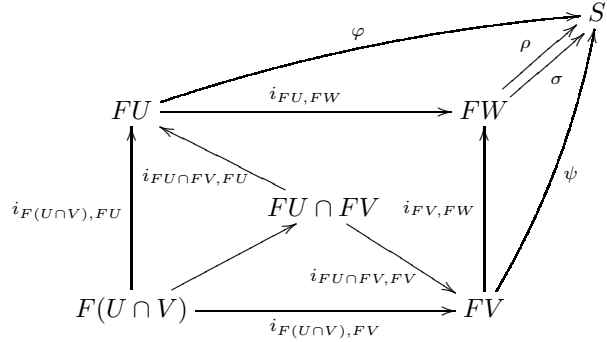
\[
\xymatrix{	& & & S \\
	FU \ar[rr]^-{i_{FU,FW}} \ar@/^/[urrr]^{\phi} & & FW \ar@<.5ex>[ur]^{\rho} \ar@<-.5ex>[ur]_{\sigma} & \\
	& FU \cap FV \ar[ul]^{i_{FU \cap FV, FU}} \ar[dr]_{i_{FU \cap FV, FV}} & \\
	F(U \cap V) \ar[ur] \ar[uu]^{i_{F(U \cap V),FU}} \ar[rr]_-{i_{F(U \cap V),FV}} & & FV \ar[uu]^-{i_{FV,FW}} \ar@/_/[uuur]_{\psi}
}
\]
commutes with the equations
\[
\sigma \circ i_{FU,FW} = \phi
\]
and
\[
\sigma \circ i_{FV, FW} = \psi
\]
holding by assumption. But then since
\[
\phi \circ i_{FU \cap FV, FU} = \psi \circ i_{FU \cap FV, FV}
\]
it follows that the diagram
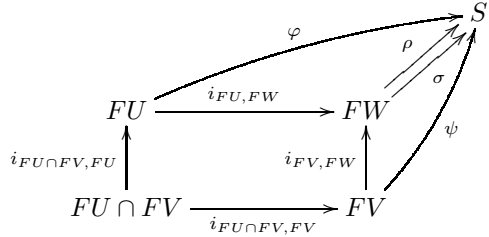
\[
\xymatrix{
	& & & S \\
FU \ar@/^/[urrr]^{\phi} \ar[rr]^-{i_{FU,FW}} & & FW \ar@<.5ex>[ur]^{\rho} \ar@<-.5ex>[ur]_{\sigma} \\
FU \cap FV \ar[u]^{i_{FU \cap FV, FU}} \ar[rr]_-{i_{FU \cap FV, FV}} & & FV \ar[u]^{i_{FV,FW}} \ar@/_/[uur]_{\psi}
}
\]
commutes as well. Thus the universal property of $FW$ as the pushout $FU \coprod_{FU \cap FV} FV$ gives that $\rho = \sigma$ and so there is a unique morphism $\rho:FW \to S$ making the diagram
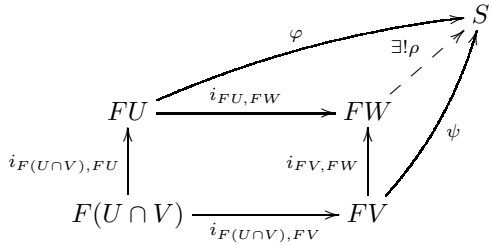
\[
\xymatrix{
 & & & S \\
FU \ar@/^/[urrr]^{\phi} \ar[rr]^-{i_{FU,FW}} & & FW \ar@{-->}[ur]^{\exists!\rho} \\	
F(U \cap V) \ar[u]^{i_{F(U \cap V),FU}} \ar[rr]_-{i_{F(U \cap V),FV}} & & FV \ar[u]^{i_{FV,FW}} \ar@/_/[uur]_{\psi}	
}
\]
commute. This shows that if $P$ is the pushout
\[
P := FU \coprod_{F(U \cap V)} FV
\]
then $FW$ is a subobject of $P$.

To prove that $P$ is a subobject of $FW$, assume that there exist morphisms $\phi^{\prime}:FU \to S^{\prime}$ and $\psi^{\prime}:FV \to S^{\prime}$ making the diagram
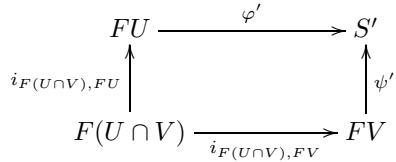
\[
\xymatrix{
FU \ar[rr]^-{\phi^{\prime}} &	& S^{\prime} \\
F(U \cap V) \ar[u]^{i_{F(U \cap V),FU}} \ar[rr]_-{i_{F(U \cap V),FV}} & & FV \ar[u]_{\psi^{\prime}} 	
}
\]
commute. Now let $\rho^{\prime}:P \to S$ be the unique map out of the pushout making
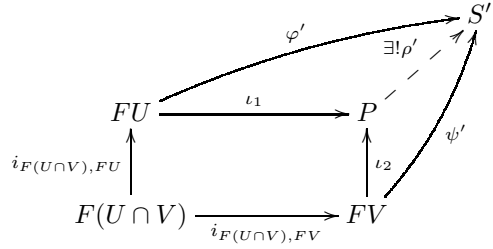
\[
\xymatrix{
	& & & S^{\prime} \\
	FU \ar[rr]^-{\iota_{1}} \ar@/^/[urrr]^{\phi^{\prime}} &	& P \ar@{-->}[ur]^{\exists!\rho^{\prime}} \\
	F(U \cap V) \ar[u]^{i_{F(U \cap V),FU}} \ar[rr]_-{i_{F(U \cap V),FV}} & & FV \ar[u]_{\iota_2} \ar@/_/[uur]_{\psi^{\prime}}
}
\]
commute. Factorizing the morphisms $i_{F(U \cap V),FU}$ and $i_{F(U \cap V),FV}$ through $FU \cap FV$ then allows us to produce the commuting diagram:
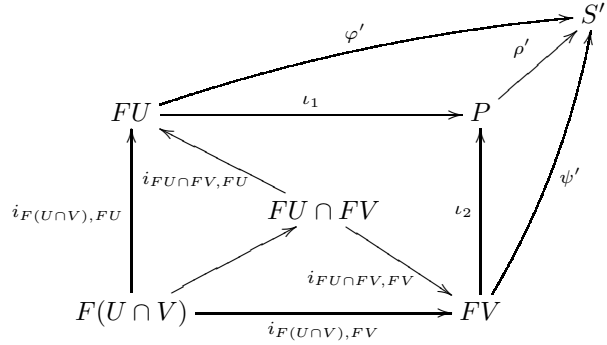
\[
\xymatrix{	& & & S^{\prime} \\
	FU \ar[rr]^-{\iota_1} \ar@/^/[urrr]^{\phi^{\prime}} & & P \ar[ur]^{\rho^{\prime}}& \\
	& FU \cap FV \ar[ul]^{i_{FU \cap FV, FU}} \ar[dr]_{i_{FU \cap FV, FV}} & \\
	F(U \cap V) \ar[ur] \ar[uu]^{i_{F(U \cap V),FU}} \ar[rr]_-{i_{F(U \cap V),FV}} & & FV \ar[uu]^-{\iota_2} \ar@/_/[uuur]_{\psi^{\prime}}
}
\]
Observe that because $P$ is the pushout  of $FU$ and $FV$ over the open subobject $F(U \cap V)$, $\iota_1$ and $\iota_2$ are both open immersions as well. Furthermore, since $F(U \cap V)$ is an open subobject of $FU \cap FV$ and since the morphisms $i_{FU \cap FV,FU}$ and $i_{FU \cap FV, FV}$ are both open immersions as well, we have that
\[
\iota_{1} \circ i_{FU \cap FV, FU} = \iota_2 \circ i_{FU \cap FV, FV}.
\]
We then compute that
\[
\phi^{\prime}  \circ i_{FU \cap FV, FU} = \rho^{\prime} \circ \iota_1 \circ i_{FU \cap FV, FU} = \rho^{\prime} \circ \iota_2 \circ i_{FU \cap FV} = \psi^{\prime} \circ i_{FU \cap FV, FV}
\]
which shows that the diagram
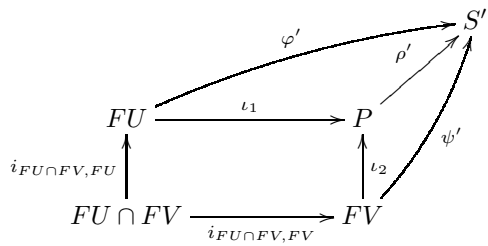
\[
\xymatrix{
	 & & & S^{\prime} \\
FU \ar[rr]^-{\iota_1} \ar@/^/[urrr]^{\phi^{\prime}} & & P \ar[ur]^{\rho^{\prime}} \\
FU 	\cap FV \ar[u]^{i_{FU \cap FV, FU}} \ar[rr]_-{i_{FU \cap FV, FV}} & & FV \ar[u]_{\iota_2} \ar@/_/[uur]_{\psi^{\prime}}
}
\]
commutes. To see that this does so universally through $\rho^{\prime}$, assume that there exists a $\sigma^{\prime}:P \to S^{\prime}$ such that
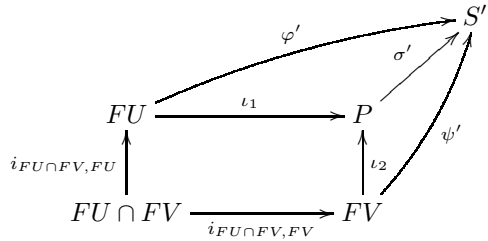
\[
\xymatrix{
	& & & S^{\prime} \\
	FU \ar[rr]^-{\iota_1} \ar@/^/[urrr]^{\phi^{\prime}} & & P \ar[ur]^{\sigma^{\prime}} \\
	FU 	\cap FV \ar[u]^{i_{FU \cap FV, FU}} \ar[rr]_-{i_{FU \cap FV, FV}} & & FV \ar[u]_{\iota_2} \ar@/_/[uur]_{\psi^{\prime}}
}
\]
commutes. However, consider now the diagram
\[
\xymatrix{
&	& & & S^{\prime} \\
&	FU \ar[rr]^-{\iota_1} \ar@/^/[urrr]^{\phi^{\prime}} & & P \ar[ur]^{\sigma^{\prime}} \\
&	FU 	\cap FV \ar[u]_{i_{FU \cap FV, FU}} \ar[rr]^-{i_{FU \cap FV, FV}} & & FV \ar[u]_{\iota_2}  	\ar@/_/[uur]_{\psi^{\prime}} \\
	F(U \cap V) \ar[ur] \ar@/_/[urrr]_-{i_{F(U \cap V),FV}} \ar@/^/[uur]^{i_{F(U \cap V),FU}} &
}
\]
and note that this composite implies that
\[
\xymatrix{
	& & & S^{\prime} \\
FU \ar@/^/[urrr]^{\phi^{\prime}} \ar[rr]^-{\iota_1} & & P \ar[ur]^{\sigma^{\prime}} \\
F(U \cap V)	\ar[u]^{i_{F(U \cap V),FU}} \ar[rr]_-{i_{F(U \cap V),FV}} & & FV \ar[u]_{\iota_2} \ar@/_/[uur]_{\psi^{\prime}}
}
\]
commutes. Using the universal property of the pushout $P$ then gives us that $\sigma^{\prime} = \rho^{\prime}$ and hence shows that $P$ is a subobject of $FW$ by a factorization of the universal property.

Now that we have shown that $FW$ is a subobject of $FU \coprod_{F(U \cap V)} FV$ and $FU \coprod_{F(U \cap V)} FV$ is a subobject of $FW$, it follows that
\[
FU \coprod_{FU \cap FV} FV = FU \cup FV = FW = FU \coprod_{F(U \cap V)} FV
\]
and so we derive from this that $FU \cap FV \cong F(U \cap V)$; however, since the map 
\[
i_{F(U \cap V), FU \cap FV}:F(U \cap V) \to FU \cap FV
\]
is an open immersion, it follows that $FU \cap FV = F(U \cap V)$ and we are done.
\end{proof}
\begin{proposition}
Let $R$ be a ring object in $\LocRingSpac$. Then the functor $h_R:\Ccat \to \LocRingSpac$ defined by, for a locally ringed space $X = (\lvert X \rvert, \CalO_X),$
\[
\lvert h_RX \rvert = \lvert X \rvert
\]
and, for opens $U \subseteq \lvert X \rvert$,
\[
\CalO_{h_RX}(U) := \Ccat(U,R)
\]
is geometrically adhesive.
\end{proposition}
\begin{proof}
The verification that this satisfies the topological side of the definition is trivial, while the sheaf-theoretic side follows from a standard gluing argument.
\end{proof}
\begin{corollary}\label{Cor: Greenberg functor is adhesive}
If $\Gr:\FSch_{/\Spec R} \to \Sch_{/\Spec k}$ is the Greenberg transform, then the left adjoint $(\mathfrak{h} \dashv \Gr^{R}):\Sch_{/\Spec k} \to \FSch_{/\Spec R}$ is geometrically adhesive.
\end{corollary}
\begin{proposition}
	If $F:\Ccat \to \Dcat$ and $G:\Dcat \to \mathscr{A}$ are geometrically adhesive, then $G \circ F:\Ccat \to \mathscr{A}$ is geometrically adhesive.
\end{proposition}
\begin{proof}
	Immediate from the calculation
	\begin{align*}
	(G \circ F)\left(\bigcup_{i \in I}U_i\right) = G\left(F\left(\bigcup_{i \in I}U_i\right)\right) = G\left(\bigcup_{i \in I}FU_i\right) = \bigcup_{i \in I} G(FU_i) = \bigcup_{i \in I} (G \circ F)(U_i).
	\end{align*}
\end{proof}

\section{The Adhesive Site}
From the definition of geometrically adhesive functors and Proposition \ref{Prop: Adhesive Functors Preserve Open Immersions}, we see that geometrically adhesive functors have good behaviour with respect to open gluings of schemes and open immersions of schemes. We would now like to define a topology on the codomain of the functor $F:\Ccat \to \Dcat$ so that the gluing condition
\[
F\left(\bigcup_{i \in I} U_i\right) \cong \bigcup_{i \in I} FU_i
\]
becomes instead an intrinsic property of the functor preserving coverages that are allowed to occur with respect to whatever topology on $\Ccat$ that we have at hand. In particular, we will show that under the construction of this topology, the Greenberg functor $\hfrak$ is site-cocontinuous (cf.\@ Corollary \ref{Corollary: The Greenberg Functor is site coconinuous}).

As we proceed, we recall briefly the notion of a sieve on a category. A sieve $S$ on an object $U$ of $\Ccat$ is a subfunctor of the representable functor on $U$, i.e., a monomorphism $S \to \Ccat(-,U)$ in the presheaf topos $[\Ccat^{\op},\Set]$. Results from classical Grothendieck topos theory (cf.\@ \cite{sga1}) show that a Grothendieck topology is completely determined by its covering sieves in a nonambiguous way. Thus we will work with sieves on the highest level possible, and specialize to working with coverages in the sense of Grothendieck pretopologies when we need to work with explicit covers.

Begin by letting $C := \lbrace \phi_i:U_i \to U \; | \; i \in I \rbrace$ be a set of covering morphisms on an object $U$ in $\Ccat$. We then define the {\em sieve generated by $C$} to be the functor $(C):\Ccat^{\op} \to \Set$ where first we define a set
\[
(C) := \lbrace \phi_i \circ \psi \; | \; \phi_i \in C, \Dom \phi_i = \Codom \psi \rbrace
\]
and then defining the action of $(C)$ on any object $X$ of $\Ccat$ via
\[
(C)(X) := \lbrace \theta \in (C) \; | \; \Dom \theta = X \rbrace,
\]
while the action of $(C)$ on morphisms is simply by precomposition. It is readily checked that $(C)$ defines a sieve on $U$. In a similar vein, if $\mathscr{C}$ is a collection of sieves we wish to make into covering sieves on a category, we write
\[
\langle C \; | \; C \in \mathscr{C} \rangle
\]
for the Grothendieck topology generated by the sieves $C$ in $\mathscr{C}$, i.e., for the minimal Grothendieck topology $J$ on $\Ccat$ which contains all the sieves $C$ as covering sieves. We will use both of these notions to define a topology on $\Dcat$ in terms of a given topology on $\Ccat$.

\begin{definition}\label{Definition: Adhesive Site}
	Let $J$ be a fixed Grothendieck topology on a category of locally ringed spaces $\Ccat$ and let $F:\Ccat \to \Dcat$ be an geometrically adhesive functor where $\Dcat$ is also a category of locally ringed spaces. We then generate a topology $\AFJ$ on $\Dcat$ by taking
	\[
	\AFJ := \langle (FC) \; | \; C\, {\rm is\, a\,} J-{\rm covering\, sieve}\rangle
	\]
	where, if $C$ is a sieve on $U \in \ob\Ccat$,
	\[
	FC := \lbrace F\phi:FV \to FU \; | \; \phi \in C \rbrace.
	\]
	The topology $\AFJ$ is then called the {\em $J$-$F$-geometrically adhesive site on $\Dcat$}.
\end{definition} 
\begin{remark}
	If the functor $F$ and the site $J$ are clear from context, we will simply refer to the site $\mathcal{A}_F^J$ as the {\em geometric adhesive site} on $\Dcat$ instead of the {\em $F$-$J$-geometric adhesive site} on $\Dcat$.
\end{remark}
\begin{lemma}\label{Lemma: geometric adhesive site is cover reflecting}
If $F:\Ccat \to \Dcat$ is a geometrically adhesive functor and if $(\Ccat,J)$ is a site, then the functor $F:(\Ccat,J) \to (\Dcat,\AFJ)$ is cover reflecting.
\end{lemma}
\begin{proof}
This is immediate from construction. Since the topology $\AFJ$ is minimally generated by $FJ$, a subfunctor $S \to \Dcat(-,FU)$ is an $\AFJ$-covering sieve on $FU$ if and only if there exists an $R \in J(U)$ such that $FR \subseteq S$. However, this is what it means to  be cover reflecting.
\end{proof}
\begin{corollary}\label{Corollary: The Greenberg Functor is site coconinuous}
The Greenberg functor $\hfrak:\Sch_{/\Spec k} \to \FSch_{/\Spec R}$ is site-cocontinuous.
\end{corollary}
\begin{proof}
From the discussion around the definition of what it means to be cover reflecting in \cite{JohnstoneSketches2}, a functor of sites is cocontinuous  if and only if it is cover reflecting.
\end{proof}

The cover reflecting property of the $\AFJ$ topology is a very practical one, as it will allow us to prove an intuitive result: It allows us to show that if a pretopology $\tau$ generates the topology $J$ on $\Ccat$, then $F\tau$ generates $\AFJ$. However, we cannot in general say that there is a pretopology generated by $F\tau$ on $\Dcat$, as $\Dcat$ need not admit pullbacks.
\begin{proposition}\label{Prop: Generated pretop generates AFJ top}
Let $F:\Ccat \to \Dcat$ be geometrically adhesive and assume that the site $(\Ccat,J)$ is generated by the pretopology $\tau$ and that $\Dcat$ admits pullbacks. Then the pretopology $\tau^{\prime} = \langle F\tau\rangle $ on $\Dcat$ generates $\AFJ$.
\end{proposition}
\begin{proof}
Observe that if $\Dcat(FU,V) = \emptyset = \Dcat(V,FU)$, then there is nothing to say, as the only covering sieves of $V$ are trivial. Moreover, by the Cover Reflecting Property it suffices to show that if  $R \in \AFJ(FU)$ is a covering sieve, then there exists a cover $C \in J(U)$ for which  $FC \subseteq R$, as we can refine along covers formed from $FU_i$'s to collections that come from functorial images of $J$-covers through the functor $F$.

To prove this let $\rho$ be the maximal pretopology on $\Dcat$ which generates $\AFJ$ and let $\lbrace \psi_i:V_i \to FU \; | \; i \in I \rbrace \in \rho(FU)$ be given such that
\[
\lbrace V_i \to FU \; | \; i \in I \rbrace \subseteq R;
\]
such a cover exists because $\rho$ generates $\AFJ$. Now, since $\AFJ$ is cover reflecting, there exists a collection of morphisms $\lbrace F\phi_j:FU_j \to FU \; | \; j \in J \rbrace$ which refines $\lbrace \psi_i:V_i \to FU \; | \; i \in I \rbrace$. However, we also have that
\[
\lbrace F\phi_j:FU_j \to FU \; | \; j \in J \rbrace = F\lbrace\phi_j:U_j \to U \; | \; j \in J\rbrace
\]
and $C := \lbrace \phi_j:U_j \to U \; | \; j \in J \rbrace \in \tau(U)$. But then it follows that  if $S$ is the $J$-covering sieve generated by this cover, $(FS) \subseteq R$. Since the cover $\langle FC \rangle \in \tau^{\prime}(FU)$, we get that $\langle FC \rangle \subseteq (FS) \subseteq R$ and so it follows that $\tau^{\prime}$ generates $\AFJ$.
\end{proof}
\begin{remark}
Note that the above proof does not show that the pretopology $\sigma = \langle F\tau \rangle$ is the {\em maximal} pretopology generating $\AFJ$. I do not think that this will happen, in general, but this remains open.
\end{remark}

Many of the results we will now give are presented in this section for the purpose of having tools to compute  the fundamental groups of the sheaf toposes $\Shv(\Dcat,\AFJ)$ so that we can have comparisons of the nature, for locally ringed spaces $U \in \ob\Ccat$ and a geometrically adhesive functor $F:\Ccat \to \Dcat$,
\[
\pi_1^{J}(U) \overset{?}\to \pi_{1}^{\AFJ}(FU), \qquad \pi_1^{\AFJ}(FU) \overset{?}\to \pi_1^{J}(U)
\]
at least in the case in which both $U$ and $FU$ are connected. We will use these in particular to study the adjunction $\mathfrak{h} \dashv \Gr^{R}:\Sch_{/\Spec k} \to \FSch_{/\Spec R}$ and to hopefully understand why, at least in the case in which $J$ is the {\'e}tale topology on $\Sch_{/\Spec k}$, passing through the Greenberg Transform allows us to geometrize quasi-characters and see things over $R$ that simply do not come from the {\'e}tale site over $\Spec R$; for details see \cite{CunningRoe} on the motivation for this idea.

In order to discuss how to move these results over, we need to prove two key lemmas which will allow us to reduce proving that a presheaf $P:\Dcat^{\op} \to \Set$ is an $\AFJ$-sheaf to checking that it satisfies the sheaf axiom on images of $J$-covers. This will then allow us to prove Lemma \ref{Lemma: Pullbacck functor preservves sheaves} below, which itself is essential for the proof of Theorem \ref{Theorem: The pullback/pushforward adjunction}. 

As we proceed, we will assume the following: The categories $\Ccat$ and $\Dcat$ have pullbacks, and if $J$ is a topology on $\Ccat$ generated by pretopology $\tau$, then $\rho = \langle F\tau\rangle$ is the pretopology on $\Dcat$ generated by $\tau$. Note that by Proposition \ref{Prop: Generated pretop generates AFJ top} $\rho$ generates $\AFJ$, so it suffices to argue if presheaves $P$ on $\Dcat$ are $\AFJ$-sheaves by checking on $\rho$-coverages by a standard result of site theory (cf.\@ Proposition 1 of \cite{MacLaneMoerdijk}). However, we need to prove the lemma below, save for with one observation at hand. Note that since $\AFJ$ is generated by $FJ$, any nontrivial cover of an object $V$ of $\Dcat$ can be refined by some cover of the form $\lbrace g_i:FU_i \to V \; | \; i \in I \rbrace$. Thus, when one considers the sheaf condition for a presheaf $P$ on $\Dcat$, if $V$ is an object of $\Dcat$ with $\Dcat(FU,V) \ne \emptyset$ for some $U$ in $\Ccat$, it suffices to consider refinements of $\lbrace V_j \to V \; | \; j \in J\rbrace$ of the form $\lbrace  FU_i \to V \; | \; i \in I \rbrace$.

%
\begin{lemma}\label{Lemma: Sufficient Check when V is not FU  but has maps from}
Let $V$ be an object of $\Dcat$ such that there exists a cover $D := \lbrace h_j:V_j \to V \; | \; j \in J \rbrace \in \rho(V)$ for which there is a refinement $\lbrace g_i:FU_i \to V \; | \; i \in I \rbrace$ that makes $\lbrace FU_i \to FX \; | \; i \in I \rbrace$ the functorial image of a cover $C := \lbrace U_i \to X \; | \; i \in I \rbrace$, where
\[
X := \bigcup_{i \in I} U_i.
\]
Then a presheaf $P$ satisfies the sheaf axiom with respect to the cover $D$ if and only if $P$ satisfies the sheaf axiom with respect to $F(C)$.
\end{lemma}
\begin{proof}
Let us begin by fixing some notation. The map
\[
e_V:P(V) \to \prod_{j \in J} P(V_j)
\]
is the pairing map $e_V = \langle P(h_j) \rangle_{j \in J}$ (and analogously $e_{FX} = \langle P(\iota_{FU_i}) \rangle_{i \in I}$), while since $\lbrace g_i:FU_i \to V \; | \; i \in I \rbrace$ refines $D$, for each $j \in J$ there exists an $i \in I$ such that there is a factorization
\[
\xymatrix{
FU_i \ar[rr]^{g_i} \ar[dr]_{\phi_{ij}} & & V \\
 & V_j \ar[ur]_{h_j}	
}
\]
in $\Dcat$. There then is a map, for any presheaf $P$ on $\Dcat$, $\alpha:\prod_{j \in J}P(V_j) \to \prod_{i \in I}P(FU_i)$ which is given as follows: For each $j \in J$, find all $i \in I$ for which there are factorizations as above, say indexed by the set $I_j$, and then construct the map
\[
\xymatrix{
\prod\limits_{j \in J} P(V_j) \ar[d]_{\pi_j} \ar[r]^-{\alpha_j} & \prod\limits_{i \in I_j} P(FU_i) \\
P(V_j) \ar[ur]_-{\langle \phi_{ij} \rangle_{i \in I_j}}
}
\]
in $\Set$; the map $\alpha$ then takes the form $\alpha = \langle \alpha_j \rangle_{j \in J}$. In this same way we can get a map $\beta:\prod_{j,j^{\prime} \in J} P(V_j \times_V V_{j^{\prime}}) \to \prod_{i, i^{\prime}} P(FU_i \times_{FU} FU_{i^{\prime}})$ by taking $\beta = \langle P(\phi_{ij} \times \phi_{i^{\prime}j^{\prime}}) \rangle$.

$\implies$: This direction is clear by performing a refinement argument and lifting any map $f:Y \to \prod_{i \in I}P(FU_i)$ along the $P\phi_{ij}$, iterating through all the $j \in J$, and then using that this lift factors through both the equalizer $e_V$ and the unique map $P(g):PV \to P(FX)$.

$\impliedby:$ We now assume that $P$ satisfies the sheaf axiom with respect to the cover $D$. Now consider the commuting diagram
\[
\xymatrix{
P(V) \ar[d]_{Pg} \ar[r]^-{e_V} & \prod\limits_{j \in J} P(V_j) \ar[d]_{\alpha} \ar@<.5ex>[rr] \ar@<-.5ex>[rr] & & \prod\limits_{ j, j^{\prime} \in J} P(V_j \times_V V_{j^{\prime}}) \ar[d]^{\beta} \\
P(FX) \ar[r]_-{e_{FX}} & \prod\limits_{i \in I} P(FU_i) \ar@<.5ex>[rr] \ar@<-.5ex>[rr] & & \prod\limits_{i, i^{\prime} \in I} P(FU_i \cap FU_{i^{\prime}})
}
\]
where the maps $\alpha, \beta, e_V,$ and $e_{FX}$ are defined as above, and $e_{FX}$ is an equalizer by assumption of $P$ satisfying the sheaf axiom on $FC$. Now suppose that there exists some set $Z$ and a morphism $f:Z \to \prod_{j \in J} P(V_j)$ such that the diagram
\[
\xymatrix{
Z \ar[r]^-{f} & \prod\limits_{j \in J} P(V_j) \ar@<.5ex>[r] \ar@<.5ex>[r] \ar@<-.5ex>[r] & \prod_{j, j^{\prime} \in J} P(V_j \times_V V_{j^{\prime}})
}
\]
commutes in $\Set$. It then follows by construction that the diagram
\[
\xymatrix{
Z \ar[r]^-{\alpha \circ f} & \prod\limits_{i \in I}	P(FU_i) \ar@<.5ex>[r] \ar@<-.5ex>[r] & \prod\limits_{i,i^{\prime} \in I} P(FU_i \times FU FU_{i^{\prime}})
}
\]
commutes in $\Set$, so by the universal property of an equalizer there exists a unique function $k:Z \to P(FX)$ making the diagram
\[
\xymatrix{
FX \ar[r]^-{e_{FX}} & \prod\limits_{i \in I}	P(FU_i) \ar@<.5ex>[r] \ar@<-.5ex>[r] & \prod\limits_{i,i^{\prime} \in I} P(FU_i \times FU FU_{i^{\prime}})	\\
V \ar@{-->}[u]^{\exists!k} \ar[ur]_{\alpha \circ f}
}
\]
commute. This implies that in particular, the diagram
\[
\xymatrix{
P(FX) \ar[r]^-{e_{FX}} & \prod\limits_{i \in I}P(FU_i) \ar@<.5ex>[r] \ar@<-.5ex>[r] & \prod\limits_{i,i^{\prime} \in I} P(FU_i \times_{FU} FU_{i^{\prime}})	\\
V \ar@{-->}[u]^{\exists! k} \ar[r]_-{f} & \prod\limits_{j \in J} P(V_j) \ar[u]_{\alpha} \ar@<.5ex>[r] \ar@<-.5ex>[r] & \prod\limits_{j,j^{\prime} \in J} P(V_j \times_V V_{j^{\prime}}) \ar[u]_{\beta}
}
\]
commutes. We now claim that $k$ factors through $P(V)$ and $P(g)$. To see this note that since
\[
\alpha \circ f = e_{FX} \circ k = \langle P(\iota_{FU_i}) \rangle_{i \in I} \circ k = \langle P\iota_{FU_i} \circ k \rangle_{i \in I}
\]
it suffices to argue on the $FU_i$ by virtue of the fact that $FX = \bigcup FU_i$ and a colimit  is determined by the maps out if its colimiting objects. Now, since for each $j \in J$ we can find an $i \in I$ such that $g_i = h_j \circ \phi_{ij}$, we need only make local gluing arguments. In particular, fix some $s \in P(FX)$ and note that $e_{FX}(s) = (s_i)_{i \in I}$ produces an element whose image in the product of pullback sections $\prod_{i,i^{\prime} \in I} P(FU_i \times_{FX} FU_{i^{\prime}})$ is the same under $P$ applied to either pullback map; call this image
\[
t = P(\pi_{1}^{i,i^{\prime}})(s_{i})_{i \in I} = P(\pi^{i,i^{\prime}}_{2})(s_i)_{i \in I} = (s_{i,i^{\prime}})_{i,i^{\prime} \in I}.
\]
However, from the refinement condition on covers and the commutativity of the diagram, we can find some section $(v_j)_{j \in J} \in \prod_{j \in J} P(V_j)$ for which
\[
\alpha(v_j)_{j \in J} =  (s_i)
\]
and hence find some $t^{\prime} = (v_{j,j^{\prime}})_{j, j^{\prime} \in J} \in \prod_{j,j^{\prime} \in J}P(V_j \times_V V_{j^{\prime}})$ which maps through $\beta$ to $t$. Explicitly,
\[
\beta(t^{\prime}) = \beta(s_{j,j^{\prime}}) = (s_{i,i^{\prime}})_{i,i^{\prime} \in I} = t.
\]
This in turn allows us to conclude that
\[
P(\pi^{jj^{\prime}}_1)(v_j) = P(\pi^{jj^{\prime}}_2)(v_j)
\]
by applying $\beta$ and then using the commutativity of the diagram. However, writing $s \in P(FX)$ as $s = k(z)$ for some $z \in Z$ gives that
\[
\alpha(v_j)_{j \in J} = (\alpha \circ f)(z) = (e_{FX} \circ k)(z)
\]
and hence the section $(s_i)_{i \in I} = e_{FX}(s)$ comes from simultaneously a uniquely given section over $FX$ and a section on each of the $V_j$. Thus we can find some $v \in P(V)$ for which 
\[
P(g)(v) = s
\]
and
\[
e_{V}(v) = (v_j)_{j \in J},
\]
by the geometry of the $V_j$ and $FX$, and the choice of these $v$ determines a function $\gamma:Z \to P(V)$. That is, we have a factorization
\[
\xymatrix{
Z \ar@/^/[drr]^{k} \ar[dr] \ar@/_/[ddr]_{f} & & \\
 & P(V)	\ar[r]^{P(g)} \ar[d]_{e_V} & P(FX) \ar[d]^{e_{FX}} \\ 
 & \prod\limits_{j \in J} P(V_j) \ar[r]_-{\alpha} & \prod\limits_{i \in I} P(FU_i)	
}
\]
which makes the diagram
\[
\xymatrix{
PV \ar[r]^-{e_V} & \prod\limits_{j \in J} P(V_j) \ar@<.5ex>[rr] \ar@<-.5ex>[rr] & & \prod\limits_{j \in J} P(V_j \times_V V_{j^{\prime}}) \\
Z \ar[u]^{\exists\gamma} \ar[ur]_{f} 	
}
\]
commutes. However, this map is $\gamma$ can be easily seen to be unique as follows: If there exists a morphism $\delta:Z \to P(V)$ giving the same factorization, the fact that $e_{FX}$ is an equalizer gives that
\[
P(g) \circ \gamma = k = P(g) \circ \delta;
\]
consequently, from the uniqueness of $k$ and the fact that $P(g)$ is determined based on the gluing data of the $FU_i$, it follows that $\gamma = \delta$. However, this implies that the diagram
\[
\xymatrix{
	PV \ar[r]^-{e_V} & \prod\limits_{j \in J} P(V_j) \ar@<.5ex>[rr] \ar@<-.5ex>[rr] & & \prod\limits_{j \in J} P(V_j \times_V V_{j^{\prime}}) \\
	Z \ar@{-->}[u]_{\exists!} \ar[ur]_{f} 	
}
\]
commutes and so $e_V$ is an equalizer, as was  to be shown.
\end{proof}
\begin{remark}
The above lemma shows that if $V$ is an object of $\Dcat$ with $V \not\cong FX$ for all objects $X$ of $\Ccat$, then we can characterize the $\AFJ$-covering sieves of $V$ as follows: Assume that the $J$-cover $\lbrace V_j \to V \; | \; j \in J \rbrace$ of $V$ is refined by $\lbrace FX_i \to V \; | \; i \in I \rbrace$ and assume that $X$ is the gluing of the $X_i$ so that $FX$ is the gluing of the $FX_i$. Let $\rho:FX \to V$ be the canonical map. Then, for any covering sieve $S$ containing the cover $\lbrace V_j \to V \; | \; j \in J \rbrace$, we must be able to find a $J$-cover $R$ of $X$ which factors through covers on each of the $X_i$, and $S$ must contain the  set $\lbrace \rho \circ F\varphi \; | \; \varphi \in R\rbrace$.
\end{remark}
\begin{corollary}\label{Cor: Suffices to prove AFJ sheaf on F of J covers}
Let $\mathscr{F}:\Dcat^{\op} \to \Set$ be a presheaf. Then $\mathscr{F}$ is an $\AFJ$-sheaf if and only if for all $U \in \ob\Ccat$, the diagram
\[
\xymatrix{
	\mathscr{F}(FU) \ar[r] & \prod\limits_{i \in I} \mathscr{F}(FU_i) \ar@<.5ex>[rr] \ar@<-.5ex>[rr] & & \prod\limits_{i,j \in I} \mathscr{F}(FU_i \cap FU_j)
}
\]
is an equalizer for all $J$-covers $\lbrace U_i \to U \; | \; i \in I \rbrace$ of $U$.
\end{corollary}
\begin{proof}
For objects of $\Dcat$ that either receive or have morphisms from or into objects $FX$, for $X \in \Ccat_0$, this follows from the above lemma. The remaining case follows from observing that if $V \in \Dcat_0$  with the relations $\Dcat(V,FX) = \emptyset = \Dcat(FX,V)$ for all $X \in \Ccat_0$, then the only $\AFJ$ covers on $V$ are trivial. This in turn implies that any presheaf satisfies the sheaf axiom over $V$; combining this with the prior lemma gives the corollary.
\end{proof}

\begin{lemma}\label{Lemma: Pullbacck functor preservves sheaves}
Let $F:\Ccat \to \Dcat$ be geometrically adhesive and let $J$ be a site on $\Ccat$. Then the functor
\[
F^{\ast}:\Shv(\Dcat,\AFJ) \to \left[(\Ccat)^{\operatorname{op}},\Set\right]
\]
given by $F^{\ast}(\mathscr{F}) := \mathscr{F} \circ F$ factors as:
\[
\xymatrix{
\Shv(\Dcat,\AFJ) \ar[dr]_{F^{\ast}} \ar[rr]^{F^{\ast}} & & \left[(\Ccat)^{\operatorname{op}},\Set\right] \\
	& \Shv(\Ccat,J) \ar[ur] & 
}
\]
\end{lemma}
\begin{proof}
Begin by observing that by Corollary \ref{Cor: Suffices to prove AFJ sheaf on F of J covers} it suffices to prove that a presheaf on $\Dcat$ is in fact an $\AFJ$-sheaf by using the functorial images of $J$-covers. Let $\mathscr{F}$ be any $\AFJ$-sheaf and let $\lbrace \phi_i:U_i \to U \; | \; i \in I \rbrace$ be a $J$-covering in a pretopology $\tau$ generating $J$. Then consider the following diagram, where the equality between the second and third rows follows from:
\[
\xymatrix{
(F^{\ast}\mathscr{F})(U) \ar[rr]^-{\langle (F^{\ast}\mathscr{F})(\phi_{i})\rangle_{i \in I}} \ar@{=}[d] & & \prod\limits_{i \in I} (F^{\ast}\mathscr{F})(U_i) \ar@{=}[d] \ar@<.5ex>[r] \ar@<-.5ex>[r] & \prod\limits_{i, j \in I} (F^{\ast}\mathscr{F})(U_i \cap U_j) \ar@{=}[d] \\
\ar@{=}[d] \mathscr{F}(FU) \ar[rr]_-{\langle \mathscr{F}(F\phi_i)\rangle_{i \in I}} & & \ar@{=}[d] \prod\limits_{i \in I} \mathscr{F}(FU_i) \ar@<.5ex>[r] \ar@<-.5ex>[r] & \prod\limits_{i,j \in I}\mathscr{F}(F(U_i \cap U_j))	\ar@{=}[d] \\
\mathscr{F}(FU) \ar[rr]_-{\langle \mathscr{F}(F\phi_i)\rangle_{i \in I}} & & \prod\limits_{i \in I} \mathscr{F}(FU_i) \ar@<.5ex>[r] \ar@<-.5ex>[r] & \prod\limits_{i,j \in I}\mathscr{F}(F(U_i) \cap F(U_j))
}
\]
Since $\AFJ$ is generated by the pretopology $\langle F\tau \rangle$ defined by the $J$-covers (cf.\@ Proposition \ref{Prop: Generated pretop generates AFJ top} above), it follows that $\lbrace F\phi_i:FU_i \to FU \; | \; i \in I \rbrace$ generates and refines an $\AFJ$-cover; using that $\mathscr{F}$ is an $\AFJ$-sheaf implies that the bottom row in the diagram is an equalizer and hence that the top is as well. This proves the lemma.
\end{proof}
\begin{Theorem}\label{Theorem: The pullback/pushforward adjunction}
The functor $F:\Ccat \to \Dcat$ induces an essential geometric morphism, perversely also named $F$,
\[
F:\Shv(\Ccat,J) \to \Shv(\Dcat,\AFJ).
\]
\end{Theorem}
\begin{proof}
We will prove that the functor $F^{\ast}$ preserves all small\footnote{Here the word ``small'' means that we assume we're working in some Grothendieck Universe $\mathscr{V}$ where all our ``small'' objects are of size $\alpha \leq \kappa$, where $\kappa$ is some strongly inaccessible cardinal. This is the last comment we will make on this subject, and the reader who does not care may simply say that anything that is ``small'' is as large as some set that is not a proper class.} colimits, as from here an appeal to Freyd's Adjoint Functor Theorem will show that $F^{\ast}$ has a right adjoint. To do this, let $\lbrace \mathscr{F}_i \; | \; i \in I \rbrace$ be a family of $\AFJ$-sheaves and let $\mathscr{F}$ be the colimit
\[
\mathscr{F} := \lim_{\longrightarrow} \mathscr{F}_i
\]
with colimit morphisms $\alpha_i:\mathscr{F}_i \to \mathscr{F}$. Now consider that to prove that $F^{\ast}$ preserves small colimits, we must show that
\[
F^{\ast}\left(\lim_{\longrightarrow} \mathscr{F}_i\right) = \lim_{\longrightarrow}\left(F^{\ast}\mathscr{F}_i\right);
\]
unwrapping this definition shows that we must prove that for all $U \in \ob\Ccat$,
\[
\left(\lim_{\longrightarrow}\mathscr{F}_i\right)(FU) \overset{?}= \lim_{\longrightarrow}\left(\mathscr{F}_i(FU)\right).
\]
To do this, we first observe that observe that since $\Shv(\Ccat,J)$ is a subtopos of the presehaf topos $\left[(\Ccat)^{\operatorname{op}},\Set\right]$, it suffices to compute whether or not the proposed equality holds by evaluating each natural transformation. However, consider that for each sheaf $\mathscr{F}_i$, there is an induced natural transformation $F^{\ast}\mathscr{F}_i  \to F^{\ast}\mathscr{F}$ given from the horizontal composition in the pasting diagram below, where $r:\Dcat \to (\Dcat)^{\operatorname{op}}$ is the formal reflection of $\Dcat$ to its opposite category:
\[
\xymatrix{
\Ccat \ar[r]^-{F} & \Dcat \ar[r]^-{r} & \Dcat^{\operatorname{op}} \rrtwocell^{\mathscr{F}_i}_{\mathscr{F}}{\alpha_i} & & \Set
}
\]
Since the sheaves $F^{\ast}\mathscr{F}_i$ and $F^{\ast}\mathscr{F}$ act on $\Ccat$ through the above diagram, we find that in the comparison diagrams that only $\alpha_i$ varies. Thus, taking the colimit we find that
\[
\lim_{\longrightarrow}(\mathscr{F}_i \circ r \circ F) = \left(\lim_{\longrightarrow}\mathscr{F}_i\right) \circ r \circ F = \mathscr{F} \circ r \circ F
\]
so we have that
\[
\lim_{\longrightarrow}\left( \mathscr{F}_i(FU)\right) = \mathscr{F}(FU) = \left(\lim_{\longrightarrow}\mathscr{F}_i\right)(FU)
\]
which proves that $F^{\ast}$ preserves all small colimits. Thus, by Freyd's Adjoint Functor Theorem, $F^{\ast}$ has a right adjoint $F_{\ast}$. Thus there is an adjunction
\[
\begin{tikzcd}
\Shv(\Dcat_{/Y},\AFJ) \ar[r,bend left,"F^{\ast}",""{name=A, below}] & \Shv(\Ccat_{/X},J)\ar[l,bend left,"F_{\ast}",""{name=B,above}] \ar[from=A, to=B, symbol=\dashv]
\end{tikzcd}
\]
which proves the first half of the essential geometric morphism.

We will now  be done if we can exhibit the existence of a left adjoint to $F^{\ast}$, i.e., if we can prove that there is a functor $F_{!}:\Shv(\Ccat,J) \to \Shv(\Dcat,\AFJ)$ such that there is an adjunction:
\[
\begin{tikzcd}
\Shv(\Ccat,J) \ar[r,bend left,"F_{!}",""{name=A, below}] & \Shv(\Dcat,\AFJ)\ar[l,bend left,"F^{\ast}",""{name=B,above}] \ar[from=A, to=B, symbol=\dashv]
\end{tikzcd}
\]
However, we can show that if 
\[
\mathscr{F} = \lim_{\substack{\longleftarrow \\ i \in I}} \mathscr{F}_i
\]
with projections $\rho_i:\mathscr{F} \to \mathscr{F}_i$, then we can calculate the limit in $\Shv(\Ccat,J)$ as in the pasting diagram
\[
\xymatrix{
	\Ccat \ar[r]^-{F} & \Dcat_{/Y} \ar[r]^-{r} & \Dcat^{\operatorname{op}} \rrtwocell^{\mathscr{F}}_{\mathscr{F}_i}{\rho_i} & & \Set
}
\]
just like the prior case. Dualizing the argument from here shows by Freyd's Adjoint Functor Theorem that $F^{\ast}$ preserves all small limits and hence that there is an adjunction
\[
\begin{tikzcd}
\Shv(\Ccat,J) \ar[r,bend left,"F_{!}",""{name=A, below}] & \Shv(\Dcat,\AFJ)\ar[l,bend left,"F^{\ast}",""{name=B,above}] \ar[from=A, to=B, symbol=\dashv]
\end{tikzcd}
\]
and hence we have the triple adjunction $F_{!} \dashv F^{\ast} \dashv F_{\ast}$ describing the essential geometric morphism $F:\Shv(\Ccat,J) \to \Shv(\Dcat,\AFJ).$
\end{proof}
\begin{remark}
Note that Theorem \ref{Theorem: The pullback/pushforward adjunction} cannot be deduced from \cite{JohnstoneSketches2} because, amongst other reasons, the functor $F$ does not preserve terminal objects in general. For instance, if $F = \hfrak:\Sch_{\Spec \Fp} \to \FSch_{/\Spec \Z_p}$, then $F(\Spec \Fp) = \Spf \Z_p \ne \Spec \Z_p$.
\end{remark}
\begin{lemma}
If $p$ is a point of $\Shv(\Ccat_{/X},J)$ then the composite geometric morphism $F\circ p$ is a point of $\Shv(\Dcat_{/Y},\AFJ)$.
\end{lemma}
\begin{proof}
Recall that the $2$-category $\Topos$ of toposes has for morphisms geometric morphisms $\CalE \xrightarrow{f} \CalF$. Then since a point of $\Shv(\Ccat_{/X},J)$ is a geometric morphism $p:\Set \to \Shv(\Ccat_{/X},J)$ and $F:\Shv(\Ccat_{/X},J) \to \Shv(\Dcat_{/Y},\AFJ)$ is an essential geometric morphism by Theorem \ref{Theorem: The pullback/pushforward adjunction}, we have that the diagram
\[
\xymatrix{
 & \Shv(\Ccat_{/X},J) \ar[dr]^{F} & \\
\Set \ar[ur]^{p} \ar[rr]_{F \circ p} & & \Shv(\Dcat_{/Y},\AFJ)
}
\]
commutes in $\Topos$. Thus $F \circ p$ is a point of $\Shv(\Dcat_{/Y},\AFJ)$.
\end{proof}
\begin{corollary}
The topos $\Shv(\FSch_{/\Spf \Z_p},\mathcal{A}_{\mathfrak{h}}^{{\acute{E}t}})$ has a point.
\end{corollary}
\begin{proof}
It is well-known that the {\'e}tale topos $\Shv(\Sch_{/\Spec \Fp},\acute{E}t)$ has points. Thus so does the sheaf topos $\Shv(\FSch_{/\Spec \Z_p},\mathcal{A}_{\mathfrak{h}}^{\acute{E}t})$ by the above Lemma.
\end{proof}

We now give some basic results about whether or not a geometrically adhesive functor induces a connected morphism of sheaf toposes. A key assumption that we will make here involves essential surjectivity of the geometrically adhesive functor $F$; this assumption may be relaxed to a condition like that given in the statement of Lemma \ref{Lemma: Sufficient Check when V is not FU  but has maps from}, but one does need to make some assumptions on the functor $F$ and the categories $\Ccat$ and $\Dcat$ in order for $F^{\ast}$ to be fully faithful. A counter-example is given below that discusses the failure of $F^{\ast}$ to be faithful more precisely, but for the connectivity results below, it simply suffices to prove that $F$ is essentially surjective.
\begin{proposition}\label{Prop: Essential Geomettric morphism is full if functor is full and surj}
Let $F:\Ccat_{/X} \to \Dcat_{/Y}$ be a full and essentially surjective geometrically adhesive functor. Then the functor $F^{\ast}$ in the essential geometric morphism $F:\Shv(\Ccat_{/X},J) \to \Shv(\Dcat_{/Y},\AFJ)$ is full.
\end{proposition}
\begin{proof}
Recall that since $F$ is full, for all $U, U^{\prime} \in \ob\Ccat_{/X}$ we have that the  map $\Ccat_{/X}(U,U^{\prime}) \to \Dcat_{/Y}(FU,FU^{\prime})$ given by $\phi \mapsto F\phi$ is epic in $\Set$. Now let $\mathscr{F}, \mathscr{G} \in \ob\Shv(\Dcat_{/Y},\AFJ)$ and consider a natural transformation $\alpha:F^{\ast}\mathscr{F} \to F^{\ast}\mathscr{G}$. Then for all $\phi:U  \to U^{\prime}$ in $\Ccat_{/X}$ the diagram
\[
\xymatrix{
(F^{\ast}\mathscr{F})(U) \ar[r]^{\alpha_U} & (F^{\ast}\mathscr{G})(U) \\
(F^{\ast}\mathscr{F})(U^{\prime}) \ar[r]_{\alpha_{U^{\prime}}} \ar[u]^{(F^{\ast}\mathscr{F})(\phi)} & (F^{\ast}\mathscr{G})(U^{\prime}) \ar[u]_{(F^{\ast}\mathscr{G})(\phi)}	
}
\]
which is equivalent to
\[
\xymatrix{
\mathscr{F}(FU) \ar[r]^{\alpha_U} & \mathscr{G}(FU) \\
\mathscr{F}(FU^{\prime}) \ar[u]^{\mathscr{F}(F\phi)} \ar[r]_{\alpha_U} & \mathscr{G}(FU^{\prime}) \ar[u]_{\mathscr{G}(F\phi)}	
}
\]
commutes in $\Set$. 

To prove that $F^{\ast}$ is full, we simply must find a lift of $\alpha$, i.e., a natural transformation $\beta:\mathscr{F} \to \mathscr{G}$ in such that $F^{\ast}\beta = \alpha$. To do this, we will construct $\beta$ in two steps: First, define $\beta_{FU}:\mathscr{F}(FU) \to \mathscr{G}(FU)$ by setting
\[
\beta_{FU} := \alpha_U.
\]
Now assume that $V \in \ob\Dcat_{/Y}$ and find a $U \in \ob\Ccat_{/X}$ such that $V \cong FU$; this is possible by the essential surjectivity of $F$. Let $\psi:FU \to V$ be any fixed isomorphism between $V$ and $FU$ and define $\beta_V$ by the rule
\[
\beta_V := \mathscr{G}(\psi^{-1}) \circ \alpha_U \circ \mathscr{F}\psi.
\]
Now let $\tilde{\phi}:V \to V^{\prime}$ be a morphism in $\Dcat_{/Y}$, where $V^{\prime} \cong FU^{\prime}$ for some $U^{\prime} \in \ob\Ccat_{/X}$ through $\psi^{\prime}:FU^{\prime} \to V^{\prime}$, and consider the commuting diagram
\[
\xymatrix{
\mathscr{F}V \ar[r]^{\mathscr{F}\psi} & \mathscr{F}(FU) \\
\mathscr{F}V^{\prime} \ar[u]^{\mathscr{F}\tilde{\phi}} \ar[r]_{\mathscr{F}\psi^{\prime}} & \mathscr{F}(FU^{\prime}) \ar@{-->}[u]_{\exists\lambda} 
}
\]
in $\Set$. Note that $\lambda$ exists because since $\mathscr{F}\psi^{\prime}$ is an isomorphism, a direct calculation with the morphism
\[
\lambda := \mathscr{F}\psi \circ \mathscr{F}\tilde{\phi} \circ \mathscr{F}(\psi^{\prime})^{-1}
\]
shows that $\lambda \circ \mathscr{F}\psi^{\prime} = \mathscr{F}\psi \circ \mathscr{F}\tilde{\phi}$. However, by the fullness of $F$, there exists a morphism $\theta:U \to U^{\prime}$ in $\Ccat_{/X}$ such that 
\[
\lambda = \mathscr{F}(F\theta) = \mathscr{F}\left( (\psi^{\prime})^{-1} \circ \tilde{\phi} \circ \psi\right).
\]
This shows that the diagram
\[
\xymatrix{
	\mathscr{F}V \ar[r]^{\mathscr{F}\psi} & \mathscr{F}(FU) \\
	\mathscr{F}V^{\prime} \ar[u]^{\mathscr{F}\tilde{\phi}} \ar[r]_{\mathscr{F}\psi^{\prime}} & \mathscr{F}(FU^{\prime}) \ar@{-->}[u]_{\exists\lambda} 
}
\]
commutes in $\Set.$ Dually, we have that the diagram
\[
\xymatrix{
\mathscr{G}(FU) \ar[r]^{\mathscr{G}(\psi^{-1})} & \mathscr{G}(V) \\
\mathscr{G}(FU^{\prime}) \ar[u]^{\mathscr{F}(F\theta)} \ar[r]_{\mathscr{G}\big((\psi^{\prime})^{-1}\big)} & \mathscr{G}V^{\prime} \ar[u]_{\mathscr{G}\tilde{\phi}}
}
\]
commutes in $\Set$; thus, since the diagram
\[
\xymatrix{
\mathscr{F}(FU) \ar[r]^{\alpha_U} & \mathscr{G}(FU) \\	
\mathscr{F}(FU^{\prime}) \ar[r]_{\alpha_{U^{\prime}}} \ar[u]^{\mathscr{F}(F\theta)} & \mathscr{G}(FU^{\prime}) \ar[u]_{\mathscr{G}(F\theta)}
}
\]
commutes by assumption, it follows that every cell in the diagram
\[
\xymatrix{
	\mathscr{F}V \ar[r]^{\mathscr{F}\psi} & \mathscr{F}(FU) \ar[r]^{\alpha_U} & \mathscr{G}(FU) \ar[r]^{\mathscr{G}(\psi^{-1})} & \mathscr{G}V \\
	\mathscr{F}V^{\prime} \ar[r]_-{\mathscr{F}(\psi^{\prime})} \ar[u]^{\mathscr{F}\tilde{\phi}} & \mathscr{F}(FU^{\prime}) \ar[r]_{\alpha_{U^{\prime}}} \ar[u]^{\mathscr{F}(F\theta)} & \mathscr{G}(FU^{\prime}) \ar[r]_{\mathscr{G}\big((\psi^{\prime})^{-1}\big)} \ar[u]_{\mathscr{G}(F\theta)} & \mathscr{G}V^{\prime} \ar[u]_{\mathscr{G}\tilde{\phi}}
}
\]
commutes and hence the outer rectangle commutes. But since we have by definition that $\beta_V = \mathscr{G}(\psi^{-1}) \circ \alpha_U \circ \mathscr{F}\psi$, it follows that the outer rectangle contracts to
\[
\xymatrix{
\mathscr{F}V \ar[r]^{\beta_V} & \mathscr{G}V \\
\mathscr{F}V^{\prime} \ar[u]^{\mathscr{F}\tilde{\phi}} \ar[r]_{\beta_{V^{\prime}}} & \mathscr{G}V^{\prime} \ar[u]_{\mathscr{G}\tilde{\phi}}	
}
\]
which proves that $\beta:\mathscr{F} \to \mathscr{G}$ is a natural transformation. Then by construction we can show that
\[
F^{\ast}\beta = \alpha,
\]
which proves that $F^{\ast}$ is full, as was desired.
\end{proof}
\begin{remark}
With the Lemma \ref{Lemma: Sufficient Check when V is not FU  but has maps from} and Corollary \ref{Cor: Suffices to prove AFJ sheaf on F of J covers} together suggest that in order to have the faithfulness of  $F^{\ast}$ that may be suggested by Lemma \ref{Lemma: Pullbacck functor preservves sheaves}, it is necessary to have the functor $F:\Ccat \to \Dcat$ satisfy the condition that for all $V \in\Dcat_0.\exists\,X\in\Ccat_0.(\Dcat(V,FX)\ne \emptyset)\lor(\Dcat(FX,V) \ne \emptyset)$. This may be made explicit by the following counter-example: 
	
Let $\Ccat = \Sch_{/\Spec \Fp}$ and let $\Dcat = \FSch_{/\Spec \Z_p} \cup \lbrace \Spec \Z_{\ell} \rbrace$, where $\gcd(\ell,p) = 1$ and with the hom-set $\Dcat(\Spec \Z_{\ell}, \Spec \Z_{\ell}) = \lbrace \id_{\Spec \Z_{\ell}}\rbrace$. Then define a functor $h:\Cat \to \Dcat$ by $h(X)  = \hfrak X$ for all $X \in \Ccat_0$ and observe that $\hfrak$ is geometrically adhesive. However, $F^{\ast}$ is not faithful as a functor of sheaf categories because you can choose any set on a presheaf over $\Spec \Z_{\ell}$ which is an $\mathcal{A}_{\hfrak}^{J}$-sheaf over the $\FSch_{/\Spec \Z_p}$ component and get an $\mathcal{A}_{h}^{J}$-sheaf. In particular, by taking the sheaves $\mathscr{S}$ and $\mathscr{T}$ to be defined by
\[
\mathscr{T}(U) = \begin{cases}
\lbrace \ast \rbrace & {\rm if}\, U \in \FSch_{/\Spec \Z_p} \\
\lbrace 0, 1 \rbrace & {\rm Else};
\end{cases}
\]
and
\[
\mathscr{S}(U) = \lbrace \ast \rbrace
\]
for all $U \in \Dcat$. Then $\Shv(\Dcat,\mathcal{A}_{h}^{J})(\mathscr{S},\mathscr{T}) = \lbrace 0,1 \rbrace$, where the labels come from which element the natural transformation picks out on the map $\lbrace \ast \rbrace \to \lbrace 0, 1 \rbrace$ (so the sections over $\Spec \Z_{\ell}$). However, since the pullback $h^{\ast}$ only sees sheaves that arise as taking $h(-)$ of $\Fp$-schemes, it is straightforward to calculate that
\[
h^{\ast}\mathscr{S}(X) = \lbrace \ast \rbrace = h^{\ast}\mathscr{T}(X)
\]
for all schemes $X$ in $\Sch_{/\Spec \Fp}$ (together with the only possible morphisms). In particular, from this construction it follows that
\[
\Shv(\Ccat,J)(h^{\ast}\mathscr{S},h^{\ast}\mathscr{T}) = \lbrace \id_{h^{\ast}\mathscr{S}} \rbrace \not\cong \lbrace 0, 1 \rbrace = \Shv(\Dcat, \mathcal{A}_h^{J})(\mathscr{S},\mathscr{T}),
\]
which shows that $h^{\ast}$ is not faithful.
\end{remark}
\begin{proposition}\label{Prop: Equiv implies that Fast is fully faithful}
If $F:\Ccat_{/X} \to \Dcat_{/Y}$ is a geometrically adhesive, full, and essentially surjective functor, then $F^{\ast}$ is fully faithful.
\end{proposition}
\begin{proof}
We have already seen from Proposition \ref{Prop: Essential Geomettric morphism is full if functor is full and surj} that $F^{\ast}$ is full; we therefore only need to show that $F^{\ast}$ is faithful. To do this, assume that there is a natural transformation $\gamma:\mathscr{F} \to \mathscr{G}$ such that $F^{\ast}\beta = F^{\ast}\gamma$ for some $\alpha:F^{\ast}\mathscr{F} \to F^{\ast}\mathscr{F}$. Since $F^{\ast}\beta = \alpha = F^{\ast}\gamma$, we have that $\gamma_{FU} = \alpha_U = \beta_{FU}$ for all $U \in \ob\Ccat_{/X}$. Now for each $V \in \ob\Dcat_{/Y}$, find a $U \in \ob\Ccat_{/X}$ and an isomorphism $\phi:FU \to V$. Then since $\beta$ and $\gamma$ are natural transformations, we have that the diagrams
\[
\xymatrix{
\mathscr{F}V \ar[d]_{\mathscr{F}\psi} \ar[r]^{\beta_V} & \mathscr{G}V \ar[d]^{\mathscr{G}\psi} \\
\mathscr{F}(FU)  \ar[r]_{\beta_{FU}}	& \mathscr{G}(FU)
}
\]
and
\[
\xymatrix{
\mathscr{F}V \ar[d]_{\mathscr{F}\psi} \ar[r]^{\gamma_V} & \mathscr{G}V \ar[d]^{\mathscr{G}\psi} \\
\mathscr{F}(FU)  \ar[r]_{\gamma_{FU}}	& \mathscr{G}(FU)	
}
\]
both commute. But then we calculate that
\begin{align*}
\beta_V &= \id_{\mathscr{G}V} \circ \beta_V = \mathscr{G}\psi^{-1} \circ \mathscr{G}\psi \circ \beta_V = \mathscr{G}\psi^{-1} \circ \beta_{FU} \circ \mathscr{F}\psi  \\ 
&= \mathscr{G}\psi^{-1} \circ \gamma_{FU} \circ \mathscr{F}\psi = \mathscr{G}\psi^{-1} \circ \mathscr{G}\psi \circ \gamma_V = \id_{\mathscr{G}V} \circ \gamma_V\\ 
&= \gamma_V
\end{align*}
so it follows that $\beta = \gamma$ and hence $F^{\ast}$ is fully faithful.
\end{proof}

We would now like to show that there is an analogous version of the above proposition that holds even in the case that $F$ is fully faithful but not essentially surjective. It will tell us that in the case in which the functor $F$ is fully faithful, we can infer that both $F^{\ast}$ reflects isomorphisms, as well as the fact that $F^{\ast}$ is fully faithful. It uses the condition stated explicitly in Lemma \ref{Lemma: Sufficient Check when V is not FU  but has maps from}: That for every object $V$ of $\Codom F$ there exists an object $X$ of $\Dom F$ for which $\Dcat(V,FX) \ne \emptyset$ or $\Dcat(FX,V) \ne \emptyset$.
\begin{proposition}\label{Prop: Fast is iso reflecting}
Assume that $F:\Ccat \to \Dcat$ is a fully faithful, geometrically adhesive functor and that for all objects $V$ of $\Dcat$ there exists an object $X$ of $\Ccat$ for which $\Dcat(FX,V) \ne \emptyset$ or $\Dcat(V,FX) \ne \emptyset$. Then $F^{\ast}$ is isomorphism reflecting.
\end{proposition}
\begin{proof}
Let $\mathscr{F}$ and $\mathscr{G}$ be $\AFJ$-sheaves on $\Dcat$ such that $F^{\ast}\mathscr{F} \cong F^{\ast}\mathscr{G}$ as $J$-sheaves on $\Ccat$. The for all objects $U$ of $\Ccat$ and for all morphisms $\phi$ of $\Ccat$ we get: From the isomorphism
\[
F^{\ast}\mathscr{F}(U) \cong F^{\ast}\mathscr{G}(U),
\]
we have that
\[
\mathscr{F}(FU) \cong \mathscr{G}(FU);
\]
and from the isomorphism
\[
F^{\ast}\mathscr{F}(\phi) \cong F^{\ast}\mathscr{G}(\phi)
\]
we also have
\[
\mathscr{F}(F\phi) \cong \mathscr{G}(F\phi).
\]
Furthermore, because both $F^{\ast}\mathscr{F}$ and $F^{\ast}\mathscr{G}$ are $J$-sheaves, it follows that for all $J$-sieves $S \in J(U)$, the diagram
\[
\xymatrix{
(F^{\ast}\mathscr{F})(U) \ar[r] & \prod\limits_{f \in S} (F^{\ast}\mathscr{F})(\Dom f) \ar@<.5ex>[rr] \ar@<-.5ex>[rr] & & \prod\limits_{\substack{f \in S, g \in \Ccat_1 \\ f \circ g \in \Ccat_1}} (F^{\ast}\mathscr{F})(\Dom g)
}
\]
is an equalizer diagram, where the $f \in S$ are sieving morphisms; similarly, the same diagram with $\mathscr{G}$ replaced with $\mathscr{F}$ is an equalizer as well. Moreover, since the definition of the $\AFJ$-topology shows that the $\AFJ$ covering sieves are generated by the sets $F(S)$, applying Corollary \ref{Cor: Suffices to prove AFJ sheaf on F of J covers} and using the hypotheses in the proposition to avoid technical malfunctions that occur off the connected component\footnote{By which we mean the connected components of the category as a graph.} generated by the essential image of $F$ allows us to check the isomorphism types of $\mathscr{F}$ and $\mathscr{G}$ along the $FS$'s. However, applying the definition of the functor $F^{\ast}$ shows that the diagram above is equivalent to the diagram
\[
\xymatrix{
\mathscr{F}(FU) \ar[r] & \prod\limits_{f \in S} \mathscr{F}(F\Dom f) \ar@<.5ex>[rr] \ar@<-.5ex>[rr] & & \prod\limits_{\substack{f \in S, g \in \Ccat_1 \\ f \circ g \in \Ccat_1}} \mathscr{F}(F\Dom g)
}
\]
Using the fully faithfulness of $F$ together with the isomoprhisms $F \Dom f \cong \Dom (Ff)$ then shows that for every set $FS$, $\mathscr{F}$ and $\mathscr{G}$ are isomorphic on the desired covers. This implies that $\mathscr{F} \cong \mathscr{G}$ and completes the proof of the proposition.
\end{proof}

We now need a lonely arithmetic geometric result for use later in this paper.
\begin{lemma}\label{Lemma: Greenberg pullback is fully faithful}
Let $R/\Z_p$ be an integral extension with residue field $k$ and let $\hfrak \dashv \Gr:\Sch_{/\Spec k} \to \FSch_{/\Spec R}$ be the Greenberg adjunction. Then the pullback functor
	\[
	\hfrak^{\ast}:\Shv(\FSch_{/\Spec R},\mathcal{A}_{\hfrak}^{J}) \to \Shv(\Sch_{/\Spec k},J)
	\]
	is fully faithful.
\end{lemma}
\begin{proof}
This follows from the fact that the topology $\mathcal{A}_{\hfrak}^{J}$ is lifted from $J$, from the fact that $\hfrak$ is a left adjoint so there is always the canonical map $\epsilon_{\mathfrak{X}}:\hfrak(\Gr(\mathfrak{X})) \to \mathfrak{X}$ for any formal scheme $\mathfrak{X}$ in $\FSch_{/\Spec R}$, and from the structure of rings of Witt vectors and algebras of Witt vectors over rings of Witt vectors.
\end{proof}

\section{The Adhesive Fundamental Group}
As a point of notation, if $\CalE$ is a Grothendieck topos, then we will write $\CalE_{lcf}$ for the full subtopos of locally constant, locally finite objects in $\CalE$. We will be making a study of these categories based on the restriction of the essential geometric morphism $F:\Shv(\Ccat,J) \to \Shv(\Dcat,\AFJ)$ induced by a geometrically adhesive functor $F:\Ccat \to \Dcat$.

Throughout this section we make the following assumptions:
\begin{center}
\begin{enumerate}
	\item[A1.] $F:\Ccat \to \Dcat$ is geometrically adhesive and the pullback $F^{\ast}$ is fully faithful;
	\item[A2.] $\Shv(\Ccat,J)_{lcf}$ is a Galois category with fibre functor $\Fi:\Shv(\Ccat,J)_{lcf} \to \FinSet$;
	\item[A3.] For every $V \in \Dcat_0$, there exists an $X \in \Ccat_0$ for which $\Dcat(V,FX) \ne \emptyset$ or $\Dcat(FX,V) \ne \emptyset$.
\end{enumerate}
\end{center} 
We will use this set up to study the fundamental group on $\Shv(\Dcat,\AFJ)$ with the provision that there is a fundamental group of $\Shv(\Ccat,J)_{lcf}$ with which we can work. This will culminate with us showing that in some situations (and in particular in the case of the Greenberg Transform) that the fundamental groups coincide. However, before we do that we must show that we can still use the techniques afforded by the essential geometric morphism which are, fittingly, essential to us.
\begin{proposition}
The functor $F^{\ast}:\Shv(\Dcat,\AFJ) \to \Shv(\Ccat,J)$ restricts to a functor $F^{\ast}:\Shv(\Dcat,\AFJ)_{lcf} \to \Shv(\Ccat,J)_{lcf}$.
\end{proposition}
\begin{proof}
Immediate from the fact that $F^{\ast}$ is exact and hence preserves all limits and colimits.
\end{proof}
\begin{proposition}
There is an essential geometric morphism $F:\Shv(\Ccat,J)_{lcf} \to \Shv(\Dcat,\AFJ)$ with pullback $F^{\ast}$ given as above.
\end{proposition}
\begin{proof}
The proof of this proposition is formal, and follows in the same way as Theorem \ref{Theorem: The pullback/pushforward adjunction}. That is, if $\lbrace\mathscr{F}_i \; | \; i \in I\rbrace$ is a finite family of functors with colimit $\mathscr{F}$ and colimit maps $\alpha_i:\mathscr{F}_i \to \mathscr{F}$, consider the following family of $2$-cells:
\[
\xymatrix{
\Ccat^{\op} \rrtwocell^{F^{\ast}\mathscr{F}_i}_{F^{\ast}\mathscr{F}}{\quad F^{\ast}\alpha_i} & & \Set
}
\]
Observing that the functor $F^{\ast}\mathscr{F}_i = \mathscr{F}(F(-)):\Ccat^{\op} \to \Set$, we can rewrite the $2$-cell as
\[
\xymatrix{
\Ccat^{\op} \rrtwocell^{\mathscr{F}_i(F^(-))}_{\mathscr{F}(F(-))}{\alpha_i} & & \Set
}
\]
which factors as
\[
\xymatrix{
\Ccat^{\op} \rrtwocell^{F^{\op}}_{F^{\op}}{\id_{F}} & & \Dcat^{\op} \rrtwocell^{\mathscr{F}_i}_{\mathscr{F}}{\beta_i} & & \Set
}
\]
and further simplifies to the diagram:
\[
\xymatrix
{
\Ccat^{\op} \ar[rr]^{F^{\op}} & & \Dcat^{\op} \rrtwocell^{\mathscr{F}_i}_{\mathscr{F}}{\alpha_i} & & \Set
}
\]
From this it follows that upon taking the colimit of the diagram that the colimiting aspect is calculated solely in the $\Shv(\Dcat,\AFJ)$ $2$-cell. Thus it follows that
\[
\lim_{\substack{\longrightarrow \\ i \in I}}F^{\ast}\mathscr{F}_i = F^{\ast}\mathscr{F};
\]
note that this is because $\Shv(\Ccat,J)_{lcf}$ and $\Shv(\Dcat,\AFJ)_{lcf}$ are the categories of locally constant, locally finite objects and morphisms between them. The fact that $F^{\ast}$ preserves limits is shown mutatis mutandis, and the fact that this implies $F^{\ast}$ admits both left and right adjoints comes from Freyd's Adjoint Functor Theorem and the fact that the codomain of $F^{\ast}$ is a topos. Finally, the fact that $F^{\ast}\mathscr{F}$ remains locally constant and locally finite comes from the fact that $F^{\ast}$ preserves all finite limits and finite colimits.
\end{proof}
\begin{lemma}
Let $\Fi:\Shv(\Ccat,J)_{lcf} \to \FinSet$ be a functor. Then:
\begin{enumerate}
	\item If $\Fi$ is exact, so is $\Fi \circ F^{\ast}: \Shv(\Dcat,\AFJ) \to \FinSet$.
	\item If $\Fi$ reflects isomorphisms, so does $\Fi \circ F^{\ast}$.
	\item If $\Fi$ is pro-representable, so is $\Fi \circ F^{\ast}$.
\end{enumerate}
\end{lemma}
\begin{proof}
Claim $(1)$ follows immediately from the fact that $F^{\ast}$ is exact because it has a left and right adjoint, and from the fact that the composite of exact functors is exact.

For part $(2)$ we assume that there are sheaves $\mathscr{F}, \mathscr{G} \in \ob\Shv(\Dcat,\AFJ)$ such that 
\[
(\Fi \circ F^{\ast})(\mathscr{F}) \cong (\Fi \circ F^{\ast})(\mathscr{G}).
\]
However, in this case because $\Fi$ reflects isomorphisms, we have that $F^{\ast}\mathscr{F} \cong F^{\ast}\mathscr{G}$, and so this reduces to the fact that $F^{\ast}$ reflects isomorphisms. However, because we have assumed properties (A1) above, the result follows from that fact that $F^{\ast}$ reflects isomorphisms and the fact that the composite of two isomorphism reflecting functors is again isomorphism reflecting.

To prove $(3)$, assume that $Fi$ is representable and let $\lbrace A_i \; | \; i \in I\rbrace$ be the filtered inverse system of objects such that
\[
\Fi(\mathscr{F}) \cong \lim_{\substack{\longrightarrow \\ I \in I}} \Shv(\Ccat,J)(A_i,\mathscr{F})
\]
for all $\mathscr{F} \in \ob\Shv(\Ccat,J)$. However, fix a $\mathscr{G} \in \ob\Shv(\Dcat,\AFJ)$ and consider that for each $i \in I$,
\[
\Shv(\Ccat,J)(A_i,F^{\ast}\mathscr{G}) \cong \Shv(\Dcat,\AFJ)(F_{!}A_i,\mathscr{G})
\]
so
\[
(\Fi\circ F^{\ast})(\mathscr{G}) = \Fi(F^{\ast}\mathscr{G}) \cong \lim_{\substack{\longrightarrow \\ I \in I}}\Shv(\Ccat,J)(A_i,F^{\ast}\mathscr{G}) \cong \lim_{\substack{\longrightarrow \\ I \in I}}\Shv(\Dcat,\AFJ)(F_{!}A_i,\mathscr{G}).
\]
Because $\lbrace A_i \; | \; i \in I \rbrace$ is a filtered inverse system, the category $I$ is filtered and so $\lbrace F_{!}A_i \; | \; i \in I \rbrace$  is a filtered inverse system by the covariance of $F_{!}$. This proves the lemma.
\end{proof}
Note that the proof of the pro-reresentability of the functor $\Fi \circ F^{\ast}$ used in an essential way the adjunction $F_{!} \dashv F^{\ast}$. Using this again in the same way we derive the lemma below:
\begin{lemma}\label{Lemma: Fibre functor pushforward by F upper star normal representing}
If $\mathscr{F} \in \ob\Shv(\Dcat,\AFJ)_{lcf}$ and if $A$ is a normal object for $\Shv(\Ccat,J)_{lcf}$ such that
\[
\Shv(\Ccat,J)(A,F^{\ast}\mathscr{F}) \cong \Fi(F^{\ast}\mathscr{F})
\]
then
\[
\Shv(\Dcat,\AFJ)(F_{!}A,\mathscr{F}) \cong \Fi \circ F^{\ast}(\mathscr{F}).
\]
\end{lemma}

\begin{lemma}\label{Lemma: F lower shriek preserves indecomposables}
Assume that $\Shv(\Dcat,\AFJ)_{lcf}$ admits an exact isomorphism reflecting functor $I:\Shv(\Dcat,\AFJ) \to \FinSet$. Then if an object $A$  in $\Shv(\Ccat,J)$ is non-initial and indecomposable, so is $F_{!}A$.
\end{lemma}
\begin{proof}
Assume that $A \in \ob\Shv(\Ccat,J)$ is indecomposable and find a nontrivial composition
\[
A \cong \coprod_{i = 1}^{n} A_i.
\]
Then if we apply $F_{!}$ to the above coproduct we get that
\[
F_{!}A = \coprod_{i = 1}^{n} F_{!}A_i,
\]
and this will be nontrivial if we can show at least one $F_{!}A_i \not\cong F_{!}A$ or $F_{!}A \not\cong \bot$. 

We will first show that at least one $F_{!}A_i \not\cong \bot$. Consider that since the decomposition of $A$ is nontrivial, there exists at least one object in the composite, call it $A_k$, such that $A_k \not\cong A$ and $A_k \not\cong \bot$. Because $A_k \not\cong \bot,$ for each object $U$ of $\Ccat$,
\[
A_k(U) \ne \emptyset
\]
as the initial sheaf $\bot$ is uniquely defined by the equation
\[
\bot(V) = \emptyset
\]
for all objects $X$ of $\Ccat$. But then we observe that $\Shv(A_k,\bot) = \emptyset$, as no nonempty set can map to the empty set and if a sheaf is empty anywhere it is empty everywhere. Now assume that there is a map $\phi:F_{!}A_k \to F_{!}\bot$; however, using that $F_{!}$ is a right adjoint shows that $F_{!}\bot= \bot$ and so there is a map $F_{!}A_k \to \bot$. Now observe that since $F^{\ast}$ is left exact we have that
\[
\Shv(\Dcat,\AFJ)(F_{!}A_k.\bot) = \Shv(\Dcat,\AFJ)(F_{!}A_k,F_{!}\bot) \cong \Shv(\Ccat,J)(A_k,F^{\ast}F_{!}\bot) = \Shv(\Ccat,J)(A_k,\bot) = \emptyset.
\]
Therefore $F_{!}A_k \not\cong \bot$ as $\Shv(\Dcat,\AFJ)(F_{!}A_k,F_{!}\bot) = \emptyset$.

We will now show that $F_{!}A_k \not\cong F_{!}A$. To do this aftpodac this were the case and note that from the assumption $F_{!}A_k \cong F_{!}A$ and from
\[
F_{!}A \cong \coprod_{i=1}^{n} F_{!}A_i
\]
we have that
\[
F_{!}A_k \cong \coprod_{i=1}^{n} F_{!}A_i = F_{!}A_k \coprod \left(\coprod_{i = 1, i \ne k}^{n} F_{!}A_i\right).
\]
Now apply $I$ to get that
\[
I\left(F_{!}A_k\right) = I(F_{!}A_k) \coprod \coprod_{i = 1, i \ne k}^{n} I(F_{!}A_i) = I(F_{!}A_k) \sqcup \bigcup_{i=1, i \ne k}^{n} I(F_{!}A_i).
\]
Now since the above equality of finite sets is a disjoint union it follows that $I(F_{!}A_i) = \emptyset$. However, because $I$ is exact, $I(\bot) = \emptyset$ and since $I$ reflects isomorphisms, this happens exactly when each $F_{!}A_i \cong \bot$. 

On the other hand, if $A_k \not\cong A$ and $A$ is indecomposable, by the above argument if follows that there is an $A_j$ for which $A_j$ is also nontrivial. Proceeding with the same line of thought and using that $F_{!}A_j \not\cong \bot$ and using that just above we have shown that $F_{!}A_j \not\cong \bot$ we have that
\[
I(F_{!}A_j)  \cong I(F_{!}A) \cong I(F_{!}A_k)
\]
while simultaneously we have
\[
I(F_{!}A_j) \sqcup I(F_{!}A_k) = I(F_{!}A).
\]
This is evidently false, and so it must be the case that $F_{!}A_k \not\cong F_{!}A$. This proves that $F_{!}A$ is indecomposable.
\end{proof}
\begin{remark}
Note that the above lemma can be false (or at least this proof fails) if we relax the assumption that the exact isomorphism reflecting functor $I$ factors through $\Set$ and not $\FinSet$, as infinite sets can be isomorphic to some of their proper subsets.
\end{remark}
\begin{remark}
It is worth remarking that the proof we gave above used the Law of the Excluded Middle. I could not find a way that did not make use of the Principle of Contraposition or Proof by Contradiction. Perhaps this should not be bothersome, however, because of the Boolean nature of Galois categories.
\end{remark}
\begin{corollary}\label{Cor: Pushforward fibre functor represented by shriek of normals}
For every sheaf $\mathscr{F}$ in $\Shv(\Dcat,\AFJ)$, there is an indecomposable sheaf $\mathscr{G}$ in $\Shv(\Dcat,\AFJ)$ such that
\[
\Fi \circ F^{\ast}(\mathscr{F})  \cong \Shv(\Dcat,\AFJ)(\mathscr{G},\mathscr{F}).
\]
\end{corollary}
\begin{proof}
Begin by observing that since each functor $F^{\ast}\mathscr{F}$ is a $J$-sheaf on $\Ccat$, there is a normal object $N$ in $\Shv(\Ccat,J)$ such that
\[
\Fi(F^{\ast}\mathscr{F}) \cong \Shv(\Ccat,J)(N,F^{\ast}\mathscr{F})
\]
by Proposition 8.46 of \cite{JohnstoneToposTheory}. Thus by Lemmas \ref{Lemma: Fibre functor pushforward by F upper star normal representing} and \ref{Lemma: F lower shriek preserves indecomposables} we have that $F_{!}N$ is indecomposable and that
\[
\Shv(\Dcat,\AFJ)(F_{!}N,\mathscr{F}) \cong \Fi \circ F^{\ast}(\mathscr{F}),
\]
which concludes the proof.
\end{proof}
\begin{lemma}\label{Lemma: Shriek of normal auto is group}
If $N$ is normal in $\Shv(\Ccat,J)_{lcf}$ then $\Shv(\Dcat,\AFJ)(F_!N,F_!N)$ is a group.
\end{lemma}
\begin{proof}
Because $N$ is normal, it is indecomposable and $N \not\cong \bot$. Thus we have that $F_!N$ by Lemma \ref{Lemma: F lower shriek preserves indecomposables} that $F_!N$ is indecomposable. Now let $\phi:F_!N \to F_!N$ be an endomorphism of $F_!N$ and consider the $(\mathcal{E},\mathcal{M})$-factorization of $\phi$:
\[
\xymatrix{
F_!N \ar[rr]^{\phi} \ar[dr]_{\epsilon} & & F_!N	\\
 & X \ar[ur]_{\mu} &
}
\]
where $X$ is a subobject of $F_!N$, $\epsilon$ is a regular epic, and $\mu$ is monic. Then since $\phi$ cannot factor through $\bot$, $X \not\cong\bot$ and hence a nonzero subobject of $F_!N$; by the indecomposability of $F_!N$ this implies that $X = F_!N$. Therefore, by considering that
\[
(\Fi \circ F^{\ast})(\epsilon):(\Fi \circ F^{\ast})(F_!N) \to (\Fi \circ F^{\ast})(F_!N)
\]
is an epimorphism from a finite set to itself, it follows that $\epsilon$ is monic and hence an automorphism of $F_!N$. Finally, since $\mu$ is monic, the map
\[
(\Fi \circ F^{\ast})(\mu):(\Fi \circ F^{\ast})(F_!N) \to (\Fi \circ F^{\ast})(F_!N)
\]
is a monomoprhim from a finite set to itself. This implies that $\mu$ is then epic and hence an automorphism as well; from here it follows that $\phi = \epsilon \circ \mu$ is also an automorphism an we are done.
\end{proof}
\begin{corollary}\label{Cor: Transfer of Normal Objects}
If $N$ is a normal object in $\Shv(\Ccat,J)_{lcf}$, $F_!N$ is normal in $\Shv(\Dcat,\AFJ)_{lcf}$.
\end{corollary}
\begin{proof}
Begin by observing that
\[
\Shv(\Ccat,J)(N,N) \cong \Fi(N)
\]
so we derive that
\[
\Shv(\Dcat,J)(F_!N,F_!N) \cong \Shv(N,F^{\ast}(F_!N)) \cong \Fi(F^{\ast}(F_!N)) \cong (\Fi \circ F^{\ast})(F_!N)
\]
because $N$ is the normal object representing $\Fi(F^{\ast}F_!N)$. Finally, by Lemma \ref{Lemma: Shriek of normal auto is group} it follows that\\ $\Shv(\Dcat,\AFJ)(F_!N,F_!N)$ is a group  and hence that $F_!N$ is normal.
\end{proof}

We now proceed to show the last ingredient in our proof of the equivalence of fundamental groups: We must know that the automorphism group of normal objects $N$ in $\Shv(\Ccat,J)_{lcf}$ is isomorphic to the automorphism group of $F_!N$.
\begin{lemma}\label{Isomorphism of normal groups}
If $N$ is a normal object in $\Shv(\Ccat,J)$ then there is an isomorphism of groups
\[
\Shv(\Ccat,J)(N,N) \cong \Shv(\Dcat,\AFJ)(F_!N,F_!N).
\]
\end{lemma}
\begin{proof}
Begin by observing that Assumption (A1) gives that $F^{\ast}$ is fully faithful; consequently the counit $\epsilon:F_!\circ F^{\ast} \to \id_{\Shv(\Dcat,\AFJ)_{lcf}}$ is an isomorphism and $F_{!}$ is essentially surjective. Then by the property of adjoint functors, there is a natural isomorphism $\theta_{F_!N}$ of hom-sets
\[
\Shv(\Dcat,\AFJ)(F_!N, F_!N) \xrightarrow{\theta_{F_!N}} \Shv(\Ccat,J)(N,(F^{\ast} \circ F_!)N);
\]
moreover, from the fact that $F^{\ast}$ is fully faithful, this isomorphism factors as
\[
\xymatrix{
\Shv(\Dcat,\AFJ)(F_!N,F_!N) \ar[rr]^-{\theta_{F_!N}} \ar[dr]_{F^{\ast}} & & \Shv(\Ccat,J)(N, (F^{\ast} \circ F_!)N) \\
	& \Shv(\Ccat,J)((F^{\ast} \circ F_!)N, (F^{\ast} \circ F_!)N) \ar[ur]_{\eta^{\ast}}
}
\]
where the map $\eta^{\ast}$ is precomposition by the unit of adjunction. Now, using that $F^{\ast}$ is a conservative exact functor between toposes of locally constant, locally finite sheaves implies that there is a Beck-Chevally condition for $(F_!,F^{\ast})$ (cf.\@ Page 179 of \cite{MacLaneMoerdijk}). This induces a natural map $(F^{\ast} \circ F_!)N \to N$ as an algebra morphism for the monad induced by $F_! \dashv F^{\ast}$, and hence gives a further isomorphism
\[
\Shv(\Ccat,J)(N, (F^{\ast} \circ F_!)N) \cong \Shv(\Ccat,J)(N,N).
\]
Composing these all gives the desired isomorphism
\[
\Shv(\Dcat,\AFJ)(F_!N,F_!N) \cong \Shv(\Ccat,J)(N,N).
\]
\end{proof}

We can now state and prove the main theorem of the  paper, and then show the geometrization result as a corollary. We reiterate the assumptions made throughout this section here for the sake of clarity.
\begin{Theorem}\label{Theorem: Isomorphism of Fundamental Groups}
Let $F:\Ccat \to \Dcat$ be a geometrically adhesive functor, suppose that $F^{\ast}$ is fully faithful,  let $X$ be an object in $(\Ccat,J)$, and let $x$ be a point of $X$ making $(\Shv(\Ccat_{/X},J_{/X})_{lcf},\Fi_{x})$ into a Galois category. Moreover, suppose that for all $V \in \Dcat$ there exists an object $X$ of $\Ccat$ such that such that $\Dcat(FX,V) \ne \emptyset$ or $\Dcat(V,FX) \ne \emptyset$. Then there is an isomorphism of fundamental groups
\[
\pi_{1}^{J}(X,x) \cong \pi_{1}^{\AFJ}(FX,Fx).
\]
\end{Theorem}
\begin{proof}
Applying the various results of this section show that there is a well-defined profinite fundamental group of $\Shv(\Dcat_{/FX}, \AFJ_{/FX})_{lcf}$ at $FX$ and $Fx$ whose normal objects are all of the form $F_!N$, where $N$ is a normal object of $\Shv(\Ccat_{/X}, J_{/X})_{lcf}$. This implies that the $F_!N$ give the correct cofinal system to limit against. Taking then the isomorphism
\[
\Shv(\Ccat,J)(N,N) \cong \Shv(\Dcat,\AFJ)(F_!N,F_!N)
\]
for all normal objects $N$, we get that
\[
\pi_{1}^{\AFJ}(FX,Fx) \cong \lim_{\longleftarrow}\Shv(\Dcat_{/X},\AFJ_{/X})(F_!N_i,F_!N_i) \cong \lim_{\longleftarrow}\Shv(\Ccat_{/X},J_{/X})(N_i,N_i) \cong \pi_1^{J}(X,x).
\]
\end{proof}
\begin{corollary}
If $G$ is a group scheme over $\Sch_{/\Spec k}$ and if $J$ is any topology on $\Sch_{/\Spec k}$ for which $\Shv(\Sch_{/Spec k},J)_{lcf}$ is a Galois category, then there is  an isomorphism of fundamental groups
\[
\pi_1^{\acute{E}t}\left(G,x\right) \cong \pi_1^{\mathcal{A}_{\hfrak}^{J}}\left(\hfrak G, \hfrak x\right).
\]
\end{corollary}

\section{Applications to Local Systems}
Here we would like to present an important corollary of Theorem \ref{Theorem: Isomorphism of Fundamental Groups} above. In particular, it says that we can geometrize quasicharacters of $p$-adic group schemes by using $p$-adic formal schemes and the $\AFJ$ topology. To see this, we let $F$ be a $p$-adic field and let $G$ be a group scheme over $F$ with geometric point $x$. Then, as is well-known to representation theorists, {\'e}tale local systems of $G$ arise as $\ell$-adic representations of the fundamental group, i.e., as representations
\[
\pi_1(G,x) \to \GL(V)
\]
where $V$ is a vector space over $\overline{\Q}_{\ell}$, for some integer prime $\ell$ coprime to $p$. Thus, if one wishes to geometrize representations of a connected, reductive group $G$ over a $p$-adic field $F$ with residue field $k$, it suffices to consider $\AFJ$-local systems on schemes of characteristic zero. This is a new result, as it allows us to (in theory, although at this point not in practice) translate some of the results in geometric representation theory to a characteristic zero analogue in formal schemes over the trait $\Spec \CalO_F$ of integers of $F$.

\begin{Theorem}\label{Theorem: Gometrization Theorem}
If $G$ is a group scheme over $\Spec k$ with geometric point $x$, then there is an isomorphism of categories
\[
\mathbf{Rep}\left(\pi_1^{\acute{E}t}(G,x)\right) \cong  \mathbf{Rep}\left(\pi_1^{\mathcal{A}_{\hfrak}^{\acute{E}t}}(\hfrak G, \hfrak x)\right).
\]
In particular, upon restricting to the categories of admissible irreducible representations, we obtain an isomorphism of categories
\[
\mathbf{AdRep}\left(\pi_1^{\acute{E}t}(G,x)\right) \cong \mathbf{AdRep}\left(\pi_1^{\mathcal{A}_{\hfrak}^{\acute{E}t}}(\hfrak G, \hfrak x)\right).
\]
\end{Theorem}
This theorem is simply a restatement of the fact that if $G \cong H$ as $p$-adic groups, then their categories of representations (and their categories of admissible representations) are isomorphic  as well.  To see how to apply this to quasicharacters of $p$-adic tori, we follow \cite{CunningRoe}. Let $F$ be a $p$-aidc field and let $\CalO_F$ be the ring of integers of $F$. Then if $T$ is a torus over $F$, it admits a N{\'e}ron model $N_T$ which is locally of finite type as a smooth, commutative group scheme over $\CalO_F$; cf.\@ \cite{NeronModels} for details. Moreover, $N_T(\CalO_F) = T(F)$ and $\Gr(N_T)(k) = N_T(\CalO_F) = T(F)$. 

In \cite{CunningRoe}, the authors showed that the Trace of Frobenius gives a natural transformation between the category of quasicharacter sheaves on $N_T$ to continuous representations $N_T(\CalO_F) \to \overline{\Q}_{\ell}^{\ast}$ which is surjective, i.e., the group homomorphism  
\[
\operatorname{Trace}(\operatorname{Frob}):\mathbf{QCS}_{/{\rm iso}}(N_T) \to \Top(\Grp)(N_T(\CalO_F),\overline{\Q}_{\ell}^{\ast})
\]
is surjective. Now, since quasicharacter sheaves on $N_T$ arise as certain local systems on $\Gr(N_T)$, which in turn are representations of $\pi_1^{\acute{E}t}(\Gr(N_t),x)$, we can use the functoriality of $\hfrak$ and the adhesive site $\mathcal{A}_{\hfrak}^{\acute{E}t}$ to lift these local systems to corresponding representations of the $\hfrak$-{\'E}t fundamental group together with the corresponding functorial lift of Frobenius. Mimicking Sections 4.5 and 4.7 of \cite{CunningRoe} then allows one to prove that the adhesive site is rich enough to geometrize quasicharacters of the torus $T$.

\section*{Acknowledgments}
I would like to give deep thanks to my PhD supervisor, Clifton Cunningham, for both giving me the problem motivating this paper, and for all his help in guiding me through the preparation of this article. I would also like to thank the organizers for the MATRIX program {\em Geometric and Categorical Representation Theory} for the kind invitation to present the work in this paper, as well as the MATRIX institute itself for the hospitality shown during the program.

\nocite{*}
\bibliography{AdhesiveBib}

\begin{bibdiv}
\begin{biblist}

\bib{RevisGreen}{article}{
      author={Bertapelle, Alessandra},
      author={Gonz\'{a}lez-Avil\'{e}s, Cristian~D.},
       title={The {G}reenberg functor revisited},
        date={2018},
        ISSN={2199-675X},
     journal={Eur. J. Math.},
      volume={4},
      number={4},
       pages={1340\ndash 1389},
         url={https://doi.org/10.1007/s40879-017-0210-0},
      review={\MR{3866700}},
}

\bib{BertapelleTong}{article}{
      author={Bertapelle, Alessandra},
      author={Tong, Jilong},
       title={On torsors under elliptic curves and {S}erre's pro-algebraic
  structures},
        date={2014},
        ISSN={0025-5874},
     journal={Math. Z.},
      volume={277},
      number={1-2},
       pages={91\ndash 147},
         url={https://doi.org/10.1007/s00209-013-1247-5},
      review={\MR{3205765}},
}

\bib{BhattScholze}{article}{
      author={Bhatt, Bhargav},
      author={Scholze, Peter},
       title={Projectivity of the {W}itt vector affine {G}rassmannian},
        date={2017},
        ISSN={0020-9910},
     journal={Invent. Math.},
      volume={209},
      number={2},
       pages={329\ndash 423},
         url={https://doi.org/10.1007/s00222-016-0710-4},
      review={\MR{3674218}},
}

\bib{Borger1}{article}{
      author={Borger, James},
       title={The basic geometry of {W}itt vectors, {I}: {T}he affine case},
        date={2011},
        ISSN={1937-0652},
     journal={Algebra Number Theory},
      volume={5},
      number={2},
       pages={231\ndash 285},
         url={http://dx.doi.org/10.2140/ant.2011.5.231},
      review={\MR{2833791}},
}

\bib{Borger2}{article}{
      author={Borger, James},
       title={The basic geometry of {W}itt vectors. {II}: {S}paces},
        date={2011},
        ISSN={0025-5831},
     journal={Math. Ann.},
      volume={351},
      number={4},
       pages={877\ndash 933},
         url={http://dx.doi.org/10.1007/s00208-010-0608-1},
      review={\MR{2854117}},
}

\bib{NeronModels}{book}{
      author={Bosch, Siegfried},
      author={L\"{u}tkebohmert, Werner},
      author={Raynaud, Michel},
       title={N\'{e}ron models},
      series={Ergebnisse der Mathematik und ihrer Grenzgebiete (3) [Results in
  Mathematics and Related Areas (3)]},
   publisher={Springer-Verlag, Berlin},
        date={1990},
      volume={21},
        ISBN={3-540-50587-3},
         url={https://doi.org/10.1007/978-3-642-51438-8},
      review={\MR{1045822}},
}

\bib{BuiumPJet}{article}{
      author={Buium, Alexandru},
       title={Geometry of {$p$}-jets},
        date={1996},
        ISSN={0012-7094},
     journal={Duke Math. J.},
      volume={82},
      number={2},
       pages={349\ndash 367},
         url={https://doi.org/10.1215/S0012-7094-96-08216-2},
      review={\MR{1387233}},
}

\bib{RobinGeoff}{article}{
      author={Cockett, J. R.~B.},
      author={Cruttwell, G. S.~H.},
       title={Differential structure, tangent structure, and {SDG}},
        date={2014},
        ISSN={0927-2852},
     journal={Appl. Categ. Structures},
      volume={22},
      number={2},
       pages={331\ndash 417},
         url={https://doi.org/10.1007/s10485-013-9312-0},
      review={\MR{3192082}},
}

\bib{RobinJohnGeoff}{article}{
      author={Cockett, J. R.~B.},
      author={Cruttwell, G. S.~H.},
      author={Gallagher, J.~D.},
       title={Differential restriction categories},
        date={2011},
        ISSN={1201-561X},
     journal={Theory Appl. Categ.},
      volume={25},
       pages={No. 21, 537\ndash 613},
      review={\MR{2861119}},
}

\bib{RobinLackIII}{article}{
      author={Cockett, Robin},
      author={Lack, Stephen},
       title={Restriction categories. {III}. {C}olimits, partial limits and
  extensivity},
        date={2007},
        ISSN={0960-1295},
     journal={Math. Structures Comput. Sci.},
      volume={17},
      number={4},
       pages={775\ndash 817},
         url={https://doi.org/10.1017/S0960129507006056},
      review={\MR{2347616}},
}

\bib{CunningRoe}{article}{
      author={Cunningham, Clifton},
      author={Roe, David},
       title={From the function-sheaf dictionary to quasicharacters of
  {$p$}-adic tori},
        date={2018},
        ISSN={1474-7480},
     journal={J. Inst. Math. Jussieu},
      volume={17},
      number={1},
       pages={1\ndash 37},
         url={https://doi.org/10.1017/S1474748015000286},
      review={\MR{3742553}},
}

\bib{Green1}{article}{
      author={Greenberg, Marvin~J.},
       title={Schemata over local rings},
        date={1961},
        ISSN={0003-486X},
     journal={Ann. of Math. (2)},
      volume={73},
       pages={624\ndash 648},
         url={https://doi.org/10.2307/1970321},
      review={\MR{0126449}},
}

\bib{Green2}{article}{
      author={Greenberg, Marvin~J.},
       title={Schemata over local rings. {II}},
        date={1963},
        ISSN={0003-486X},
     journal={Ann. of Math. (2)},
      volume={78},
       pages={256\ndash 266},
         url={https://doi.org/10.2307/1970342},
      review={\MR{0156855}},
}

\bib{sga1}{book}{
      author={Grothendieck, Alexander},
       title={Rev\^etements \'etales et groupe fondamental ({SGA} 1)},
      series={Lecture notes in mathematics},
   publisher={Springer-Verlag},
        date={1971},
      volume={224},
}

\bib{JohnstoneToposTheory}{book}{
      author={Johnstone, Peter~T.},
       title={Topos theory},
     edition={Dover Edition},
   publisher={Academic Press, Inc.},
     address={New York},
        date={1977},
        note={{D}over {P}ublications, {I}nc. 2014 republication},
}

\bib{JohnstoneSketches2}{book}{
      author={Johnstone, Peter~T.},
       title={Sketches of an elephant: a topos theory compendium. {V}ol. 2},
      series={Oxford Logic Guides},
   publisher={The Clarendon Press, Oxford University Press, Oxford},
        date={2002},
      volume={44},
        ISBN={0-19-851598-7},
      review={\MR{2063092}},
}

\bib{AdhesiveQuasiAdhesive}{article}{
      author={Lack, Stephen},
      author={Soboci\'{n}ski, Pawe\l},
       title={Adhesive and quasiadhesive categories},
        date={2005},
        ISSN={0988-3754},
     journal={Theor. Inform. Appl.},
      volume={39},
      number={3},
       pages={511\ndash 545},
         url={https://doi.org/10.1051/ita:2005028},
      review={\MR{2157046}},
}

\bib{Lang}{article}{
      author={Lang, Serge},
       title={On quasi algebraic closure},
        date={1952},
        ISSN={0003-486X},
     journal={Ann. of Math. (2)},
      volume={55},
       pages={373\ndash 390},
         url={https://doi.org/10.2307/1969785},
      review={\MR{0046388}},
}

\bib{MacLaneMoerdijk}{book}{
      author={Mac~Lane, Saunders},
      author={Moerdijk, Ieke},
       title={Sheaves in geometry and logic},
      series={Universitext},
   publisher={Springer-Verlag, New York},
        date={1994},
        ISBN={0-387-97710-4},
        note={A first introduction to topos theory, Corrected reprint of the
  1992 edition},
      review={\MR{1300636}},
}

\bib{NicaiseSebagMSIandWR}{article}{
      author={Nicaise, Johannes},
      author={Sebag, Julien},
       title={Motivic {S}erre invariants and {W}eil restriction},
        date={2008},
        ISSN={0021-8693},
     journal={J. Algebra},
      volume={319},
      number={4},
       pages={1585\ndash 1610},
         url={https://doi.org/10.1016/j.jalgebra.2007.11.006},
      review={\MR{2383059}},
}

\bib{NicaiseSebagMIand}{incollection}{
      author={Nicaise, Johannes},
      author={Sebag, Julien},
       title={Motivic invariants of rigid varieties, and applications to
  complex singularities},
        date={2011},
   booktitle={Motivic integration and its interactions with model theory and
  nonarchimedean geometry},
      editor={Cluckers, R.},
      editor={Nicaise, J.},
      editor={Sebag, J.},
      series={London Mathematical Society Lecture Notes Series},
      volume={383},
   publisher={Cambridge University Press},
       pages={244 \ndash  304},
}

\bib{SchwedeClosedGlue}{incollection}{
      author={Schwede, Karl},
       title={Gluing schemes and a scheme without closed points},
        date={2005},
   booktitle={Recent progress in arithmetic and algebraic geometry},
      series={Contemp. Math.},
      volume={386},
   publisher={Amer. Math. Soc., Providence, RI},
       pages={157\ndash 172},
         url={https://doi.org/10.1090/conm/386/07222},
      review={\MR{2182775}},
}

\bib{Yu}{incollection}{
      author={Yu, Jiu-Kang},
       title={Smooth models associated to concave functions in {B}ruhat-{T}its
  theory},
        date={2015},
   booktitle={Autour des sch\'{e}mas en groupes. {V}ol. {III}},
      series={Panor. Synth\`eses},
      volume={47},
   publisher={Soc. Math. France, Paris},
       pages={227\ndash 258},
      review={\MR{3525846}},
}

\end{biblist}
\end{bibdiv}

\end{document}